\documentclass[11pt]{amsart}
\usepackage{amsmath,amsthm, amscd, amssymb, amsfonts, mathtools}
\usepackage[all]{xy}
\usepackage{enumitem}
\usepackage{hyperref}
\usepackage{t1enc}
\usepackage[latin1]{inputenc}
\usepackage{fontsmpl}
\usepackage{srcltx}
\usepackage{placeins}
\usepackage{float}
\usepackage[dvipsnames]{xcolor}
\usepackage[normalem]{ulem}

\allowdisplaybreaks

\numberwithin{equation}{section}\theoremstyle{plain}

\newtheorem{thm}{Theorem}[section]
\newtheorem{prop}[thm]{Proposition}
\newtheorem{lem}[thm]{Lemma}
\newtheorem{cor}[thm]{Corollary}
\theoremstyle{definition}

\newtheorem{rem}[thm]{Remark}

\newcommand\I{\mathbb I}

\def \N {\mathbb{N}}
\def \Z {\mathbb{Z}}
\def \k {\Bbbk}
\def \o {\otimes}

\def \s {\sigma^{-1}}

\begin{document}


 \title[On Partial actions of Hopf-Ore extensions]{On Partial actions of Hopf-Ore extensions}
 
\author[Giraldi, Martini and Silva]{Jo\~{a}o M. J. Giraldi, Grasiela Martini and Leonardo D. Silva}

\address[Giraldi]{Universidade Federal do Rio Grande do Sul, Brazil}
\email{joao.giraldi@ufrgs.br}

\address[Martini]{Universidade Federal do Rio Grande do Sul, Brazil}
\email{grasiela.martini@ufrgs.br}

\address[Silva]{Universidade Federal do Rio Grande do Sul, Brazil.}
\email{dsleonardo@ufrgs.br}

\thanks{\noindent \textbf{2020 MSC:} Primary 16T05; Secondary 16T99. \\ 
	\hspace*{0.3cm}
	\textbf{Key words and phrases:} Partial action; Hopf-Ore extension; Rank one Hopf
algebra.	\\
The second and third author were partially supported by FAPERGS (Brazil), projects n. 23/2551-0000803-3 and 23/2551-0000913-7, respectively.}

\begin{abstract}
In this work we study how to extend a partial action of a Hopf Algebra $A$ on an algebra $R$ to a partial action of a Hopf-Ore extension of $A$ on $R$.
As consequence, we characterize all partial actions of rank one Hopf
algebras (in particular, generalized Taft algebras and Radford algebras), under suitable conditions.
\end{abstract}

\maketitle

\tableofcontents

\section{Introduction}

In the last years, the theory of partial actions and representations is a subject widely studied in several contexts such as groups \cite{Dokuchaev_exel}, groupoids \cite{paques_bagio}, (weak) Hopf algebras \cite{enveloping, caenepeel2008partial, glauber_e_fefi}, categories \cite{ hopf_categories}, among others.
For partial actions of Hopf algebras, despite of the developed theory, the computations to characterize  explicit examples are a challenging work. 

Recently, this question has been addressed in \cite{FMS, FMS2, corresponding}.
In these works, the applied method to get the partial actions of some families of Hopf algebras usually lies in extending the partial actions of Hopf subalgebras.
Namely, some partial actions of pointed Hopf algebras are computed extending the partial actions of their group of grouplike elements.
This technique is the same used as in the global case (see, for instance \cite{Acoes_taft_Centrone, Acoes_taft_Chelsea}).

In this work, we approach this question in a more general way.
Since the families of Hopf algebras treated in the above references are quotients of Hopf-Ore extensions, we investigate when is possible to extend the partial action of a Hopf algebra to a Hopf-Ore extension. In particular, Theorem \ref{formula_eh_acao_parcial_volta} shows that, under suitable conditions, the partial actions still depend only on the partial action of the base Hopf algebra, a skew-derivation  and a compatibility between them.

We highlight that the techniques and results developed here are  effective tools. Indeed, in this paper we characterize some partial actions of interesting families of Hopf algebras, as the rank one Hopf algebras (see \S \ref{Sec_Rank_one}); such classification  is an original work to the best of authors' knowledge.
On the other hand, this (more refined) approach also recovers the main results within \cite{FMS, FMS2, corresponding}.

This work is organized as follows. In \S 2, the basic concepts and results of q-combinatorics, partial actions and Hopf-Ore extensions are briefly recalled.

In \S \ref{sec:3}, we begin determining a condition for extending a partial action of a Hopf algebra $A$ to a partial action of a Hopf-Ore extension $A[x,\sigma,\delta]$ in the trivial way.
In the sequence we obtain necessary and sufficient conditions to extend a partial action of $A$ to $A[x,\sigma]$ when $g\cdot 1_R=0$. Further, in \S \ref{Sec_quotient}, we investigate the relations between the partial actions of Hopf-Ore extensions and the partial actions of its  quotients.

In the last section, we present applications and examples. 
In Sections  \ref{Sec_xcentral} and \ref{Sec_A=KG}, we improve the sufficient/necessary conditions within Theorem \ref{formula_eh_acao_parcial_volta} for the following two cases:  $x\cdot 1_R\in Z(R)$ and  $A=\Bbbk G$, respectively.
In \S \ref{Sec_Rank_one}, we  obtain all partial actions of rank one Hopf algebras (in particular,  generalized Taft algebras and Radford algebras) on an algebra $R$ under (not so restrictive) suitable conditions.
 At last, we explore the case for Hopf-Ore extensions of Sweedler's Hopf algebra and Nichols Hopf algebras, respectively in \S \ref{Sec_Sweedler} and \S \ref{Sec_Nichols}.

\subsection*{Conventions and notations}\label{Sec_convention_notation}

Throughout the paper, we work over an algebraically closed field $\k$ of characteristic zero. The group of multiplicative invertible elements of $\k$ is denoted by $\Bbbk^{\times}$.
For $n\geq 1$, $\mathbb{G}_n\subset \Bbbk^{\times} $ is the group of $n$-th roots of unity; by $\mathbb{G}'_n \subset \mathbb{G}_n$, we mean the subset of primitive roots.

For the positive integers, we use the notation $ \mathbb{N} $, while $\mathbb{N}_0 = \mathbb{N}\cup \{ 0\}$.
If $j \leq k\in\mathbb{Z}$, then we consider $\I_{j,k} = \{j, j + 1, \cdots , k\}$ and $\I_k = \I_{1, k}$. 

Let $\mathcal{H}$ be a Hopf algebra.
We denote the group of its grouplike elements by $G(\mathcal{H})$.
Given $h, g \in G(\mathcal{H})$, then an element $x$ is called a $(h,g)$-primitive if $\Delta(x) = x\o h + g\o x$.
The space of all $(h,g)$-primitive elements is 
$P_{h,g}(\mathcal{H})$.
Moreover, $Z(R)$ is the center of a general algebra $R$.

\section{Preliminaries}

\subsection{q-combinatorics}\label{Sec_qcomb}
For this section, we refer  \cite[\S 7.2]{livro_radford}. 
Let $n, m \in \Z$, $n \geq 0$, and $q \in \k^{\times}$. The \emph{q-binomial coefficients} are defined recursively as follows.
First, set ${0 \choose 0}_q = 1$ and ${n \choose m}_q = 0$  if $m>n$ or $m<0$.
Then, for $n \geq 1$ and $0 \leq m \leq n$,
\begin{align} \label{6.7}
	{n \choose m}_q  =  {{n-1} \choose m}_q  + q^{n-m}  {{n-1} \choose {m-1}}_q.
\end{align}
As a consequence of the definition, observe that 
\begin{align}\label{id_q}
	{n \choose m}_q = q^{m(n-m)} { n \choose m}_{q^{-1}}.
\end{align}

Assume that $q\in \mathbb{G}'_N$.  For $k \geq 0$, let $k_D$ and $k_R$ be the integers determined by $k =k_D \, N + k_R$ where $0 \leq k_R <N$.
Then 
\begin{align}\label{radford}
	{n \choose m}_q  = {n_R \choose m_R}_q {n_D \choose m_D},
\end{align}
for all $0 \leq m \leq n$.

\begin{lem} \cite[Lem. 2.2 and 2.3]{FMS2} 
Let $q \in \k^{\times}$ and $i, j, k\in \N_0$.
	Then, 
 \begin{eqnarray}
     &&{ j \choose k }_q { {j-k} \choose {i-k} }_q = { j \choose i }_q { i \choose k }_q; \label{Lema2_2}\\
     && \sum_{s=0}^{j} (-1)^{j-s} q^{\frac{s(s+1)}{2}-sj} {j \choose s}_q {{s+i} \choose {k}}_q = \, q^{\frac{j(j+1)}{2} + j(i-k)} {i \choose k-j}_q. \label{lema_acao_da_acao} 
      \end{eqnarray}
  \end{lem}

\subsection{Partial actions}
A \emph{(left) partial action of a Hopf algebra (or a bialgebra) $\mathcal{H}$ on an algebra $R$} is a linear map $\cdot: \mathcal{H} \otimes R \to R$, denoted by $\cdot (h \o r) = h \cdot r$, such that
\begin{itemize}
	\item[(PA.1)]\label{PA1} $1_{\mathcal{H}} \cdot r=r$, 
	\item[(PA.2)]\label{PA2} $h\cdot rs=(h_1\cdot r)(h_2\cdot s)$,
	\item[(PA.3)]\label{PA3} $h\cdot(k\cdot r)=(h_1\cdot 1_R)(h_2k\cdot r)$,
\end{itemize}
hold for all $h,k\in \mathcal{H}$ and $r,s\in R$.
In this case, $R$ is called a \textit{(left) partial $\mathcal{H}$-module algebra}.
A (left) partial action is \emph{symmetric} if in addition we have
\begin{itemize}
	\item[(PA.S)]\label{PAS} $h \cdot ( k \cdot r)=(h_1k \cdot r)(h_2 \cdot 1_R)$.
\end{itemize}

Observe that, if (PA.1) holds, then conditions (PA.2) and (PA.3) are equivalent  to the following one:
\begin{itemize}
	\item[(PA.2$^\prime$)]\label{PA2e3} $h \cdot (r( k \cdot s))=(h_1 \cdot r)(h_2k \cdot s)$.
\end{itemize}

\begin{lem}\label{ide_central}\cite[Lem. 3.1]{FMS2} Let $\cdot : \mathcal{H}\otimes R \to R$ be a partial action. For each $g\in G(\mathcal{H})$ and $r\in R$,
\begin{eqnarray}\label{ide_quasecentral}
    (g\cdot 1_R)r (g\cdot 1_R)= (g\cdot 1_R)r.
\end{eqnarray}
Moreover, if $\cdot$ is symmetric, then $(g\cdot 1_R)\in Z(R)$.
\end{lem}

Notice that if $h \in \mathcal{H}$ satisfies $\Delta(h)=\tau \Delta (h)$ (where $\tau$ is the usual twist morphism), then it holds in general that  $(h_1 \cdot 1_R ) r (h_2 \cdot 1_R)= (h\cdot 1_R) r,$ and moreover $h \cdot 1_R \in Z(R)$ if the partial action $\cdot$ is symmetric.

\medbreak

Now we present some lemmas that are useful for the next sections.

\begin{lem}
 \label{parte_global}
	Let $\cdot : \mathcal{H} \o R \to R$ be a partial action, $g \in G(\mathcal{H})$ and $x \in P_{1,g}(\mathcal{H})$. If $g \cdot 1_R = 1_R$, then $g$ and $x$ act globally, i.e., $g \cdot (h \cdot r) = gh \cdot r$ and $x \cdot (h \cdot r) = xh \cdot r$, for all $h\in \mathcal{H}, r\in R$.
   \end{lem}

\begin{proof}
Suppose that $g \cdot 1_R= 1_R$.
Then, for all $h\in \mathcal{H}, r\in R$,
$$g \cdot (h \cdot r) = (g \cdot 1_R) (gh \cdot r) = gh \cdot r.$$

Also,
$x \cdot 1_R = x \cdot 1_R^2 =  (x \cdot 1_R)(1 \cdot 1_R) + (g \cdot 1_R)(x \cdot 1_R)= 2 (x \cdot 1_R),$
that is, $x \cdot 1_R = 0$.	
Hence, 
	$$x \cdot (h \cdot r) = (x \cdot 1_R)(h \cdot r) + (g \cdot 1_R)(xh \cdot r) = xh \cdot r.$$
	\end{proof}

\begin{lem}\label{acao_produto_de_skew_primitivos_nichols}
Let $\cdot : \mathcal{H} \otimes R \to R$ be a partial action, $g \in G(\mathcal{H})$ and $x\in P_{1,g}(\mathcal{H})$ such that $xg=qgx$, for some $q \in \Bbbk^{\times}$. 
If $g \cdot 1_R =0$,
then
\begin{itemize}
\item[(i)] $x \cdot r = (x \cdot 1_R)r$ and $g^{-1}x \cdot r = -q r (x \cdot 1_R)$,  for all $r \in R$, and 
\item[(ii)] if $g^2=1$, $q=-1$ and $x^2=0$, then $(x \cdot 1_R)^2 \in Z(R)$.
\end{itemize}
\end{lem}

\begin{proof} Assume $g \in G(\mathcal{H})$ and $g \cdot 1_R =0$.
Thus, $g\cdot r=g^{-1}\cdot r=0$, for all $r\in R$. 

$(i)$ Since $x \in P_{1,g}(\mathcal{H})$, we have $g^{-1}x \in P_{g^{-1}, 1}(\mathcal{H})$.
Then $x\cdot r=x\cdot (1_Rr)=(x\cdot 1_R)r$ and $g^{-1}x\cdot r=g^{-1}x\cdot (r 1_R)=r(g^{-1}x\cdot 1_R)$. Moreover, 
$$0=g^{-1}x\cdot(g\cdot 1_R)= (g^{-1}x\cdot 1_R)+ (g^{-1}xg\cdot 1_R).$$
Then, $g^{-1}x\cdot 1_R=-q(x\cdot 1_R)$ and, therefore, $g^{-1}x\cdot r=-qr(x\cdot 1_R)$.

$(ii)$ Assume that $g^2=1, q=-1$ and $x^2=0$.
Thus $gx\cdot(gx\cdot r)=(gx\cdot 1_R)(x\cdot r)$. 
By item (i), it follows that $r(x\cdot 1_R)^2=(x\cdot 1_R)^2r$, for all $r\in R$.
\end{proof}

\begin{lem}\label{acao_produto_de_skew_primitivos_nichols2}
Let $\cdot : \mathcal{H} \otimes R \to R$ be a partial action, $g \in G(\mathcal{H})$ and  $x,y\in P_{1,g}(\mathcal{H})$ such that $g^2=1$, $xg=-gx$, $yg=-gy$.
Then,
\begin{itemize}
\item[(i)] if $(x \cdot 1_R), (y \cdot 1_R) \in Z(R)$ and $g\cdot 1_R=0$, then $xy \cdot r = gxy \cdot r = 0$, for all $r\in R$,
\item[(ii)] if $h\in \mathcal{H}$ satisfies $h \cdot r = gh\cdot r = 0$, for some $r\in R$, then $xh \cdot r = gxh \cdot r = 0$.
\end{itemize}
\end{lem}

\begin{proof}
First observe that, for all $r\in R$, we have
$gx\cdot(y\cdot r)=(gx\cdot 1_R)(gy\cdot r)+(gxy\cdot r)$ and $gx\cdot(gy\cdot r)=(gx\cdot 1_R)(y\cdot r)-(xy\cdot r).$
Then, by Lemma \ref{acao_produto_de_skew_primitivos_nichols}, 
\begin{align*}
    gxy\cdot r &= (y\cdot 1_R)r(x\cdot 1_R)-(x\cdot 1_R)r(y\cdot 1_R) \quad\quad \text{and} \\
    xy\cdot r &= (x\cdot 1_R)(y\cdot 1_R)r-r(y\cdot 1_R)(x\cdot 1_R).
\end{align*}
Since $(x \cdot 1_R), (y \cdot 1_R) \in Z(R)$, it follows that $xy\cdot r=gxy\cdot r=0$, for all $r\in R$, and it concludes item (i).

For (ii), assume that $h \cdot r = gh \cdot r = 0$.
Hence, 
$gx\cdot(gh\cdot r)=(gx\cdot 1_R)(h\cdot r)-xh\cdot r$ and $gx\cdot(h\cdot r)=(gx\cdot 1_R)(gh\cdot r)+gxh\cdot r$, and so $xh \cdot r = gxh \cdot r = 0$. 
\end{proof}

\subsection{Hopf-Ore extensions}

Let $A$ be an algebra, and consider an automorphism $\sigma$ of $A$ and a $\sigma$-derivation $\delta$ of $A$, \emph{i.e.}, a linear map $\delta: A \to A$ such that $\delta(ab)=\delta(a)b+\sigma(a)\delta(b)$.
The \emph{Ore extension $A[x,\sigma,\delta]$ of the algebra $A$} is the algebra  generated by $A$ and a variable $x$, subjected to the relation
\begin{equation}\label{prod_ore}
    xa=\sigma(a)x+\delta(a),
\end{equation}
for all $a\in A$. Every element of $A[x,\sigma,\delta]$ can uniquely be represented as $\sum_{i \geq 0}a_ix^i$, where $a_i \in A$.

In \cite{Panov}, Panov determined necessary and sufficient conditions for an Ore extension $A[x,\sigma,\delta]$ over a Hopf algebra $A$ to have a structure of a Hopf algebra which extends the Hopf algebra structure of $A$ and such that $x$ is a skew-primitive element, that is, $\Delta(x)=x \o 1 + g \o x$, for some $g \in G(A)$. In this case, $A[x,\sigma,\delta]$ is called a \emph{Hopf-Ore extension of $A$}. 

\begin{thm} \cite[Thm. 1.3]{Panov}  \label{teo_panov} The Hopf algebra $A[x,\sigma,\delta]$ is a Hopf-Ore extension if only if
\begin{itemize}
    \item[(i)] there exists a character $\chi: A\to \Bbbk$ such that $\sigma(a) = \chi(a_1)a_2 = ga_1g^{-1}\chi(a_2)$,
    \item[(ii)] the $\sigma$-derivation $\delta$ satisfies $\Delta(\delta(a))=\delta(a_1)\otimes a_2+ga_1\otimes\delta(a_2)$.
\end{itemize}
    \end{thm}

\begin{rem} \label{remark_cond_hopf ore} Note that, as $\sigma$ is an automorphism, $ax=x\s(a)-\delta(\s(a))$ and $\sigma^{-1}(a)=\chi^{-1}(a_1)a_2=g^{-1}a_1g\chi^{-1}(a_2)$, for all $a\in A$. Moreover, for all $i\in\mathbb{Z}$, it holds that
$\Delta(\sigma^{-i}(a)) = \sigma^{-i}(a_1) \otimes a_2$  and $\Delta(\sigma^{-i}(a)) = g^{-i}a_1g^i \otimes \sigma^{-i}(a_2)$.
Also, $\sigma^{-i}(a)=g^{-i}a_1g^{i}\chi^{-i}(a_2)$ and $\sigma^{-i}(g^k)=\left(\chi(g)\right)^{-ik}g^k$.
\end{rem}

\section{Extensions of partial actions via Hopf-Ore extensions} \label{sec:3}

In this section, we  characterize some partial actions of a Hopf-Ore extension $A[x,\sigma,\delta]$ on $R$. 
For a Hopf-Ore extension, the element $g \in G(A)$ plays a crucial role since $\Delta(x)=x \otimes 1 + g \otimes x$. 
Likewise, we check that the action of the element $x$ is closely related to the action of the element $g$.

Initially, we prove a case when it is possible to extend a partial action of $A$ on $R$ to a partial action of $A[x,\sigma,\delta]$ on $R$ in a trivial way.

\begin{prop}
      Let $\cdot : A \o R \to R$ be a (symmetric) partial action of a Hopf algebra $A$ on an algebra $R$.
Given a Hopf-Ore extension $A[x,\sigma,\delta]$, consider $J$ the ideal of $A$ generated by $\delta(A)$.
If $J\cdot R=0$, then the map  $\triangleright : A[x,\sigma,\delta] \otimes R \to R$, defined as $x^ja \triangleright r = \delta_{j,0} (a \cdot r)$, for all $a \in A$ and $r \in R$, is a (symmetric) partial action. 
\end{prop}

\begin{proof}
Suppose $J\cdot R=0$. By \eqref{prod_ore} we have $ax^i = \sum_{\ell = 0 }^i x^\ell B^i_\ell(a)$, for $B^i_\ell ( a ) \in A$, where $B^i_i(a)=\sigma^{-i}(a)$ and $B^i_0(a)=(-1)^i (\delta\circ\sigma^{-1})^i(a)$, for all $a \in A$. 
Moreover, since $\Delta(x^ja)=x^ja_1\otimes a_2+\sum_{k=1}^{j} { j \choose k }_{q^{-1}} x^{j-k} g^k a_1 \otimes x^k a_2$, on the one hand  we get
\begin{eqnarray*}
& & ( (x^ja)_1 \triangleright r) ((x^ja)_2 x^ib \triangleright s) \\
&=& (x^ja_1  \triangleright r)(a_2x^ib \triangleright s) + \sum_{k=1}^{j} { j \choose k }_{q^{-1}} ( x^{j-k} g^k a_1 \triangleright r)(x^k a_2x^ib \triangleright s)\\
     &=& (x^ja_1 \triangleright r) \left( \sum_{\ell = 0 }^i \left(x^\ell B^i_\ell(a_2) b \triangleright s \right) \right) \\
    & & + \sum_{k=1}^{j} { j \choose k }_{q^{-1}} (x^{j-k} g^k a_1 \triangleright r) \left( \sum_{\ell = 0 }^i \left( x^{k+\ell}  B^i_\ell (a_2) b \triangleright s \right) \right)\\
    &=& \delta_{j,0} (a_1\cdot r) \left( \sum_{\ell = 0 }^i \delta_{\ell,0} 
 (B^i_\ell (a_2) b \cdot s) \right)\\
    & & + \sum_{k=1}^{j} { j \choose k }_{q^{-1}} \delta_{j-k,0}(g^k a_1  \cdot r)\left( \sum_{\ell = 0 }^i \delta_{k+\ell,0} \left(B^i_\ell(a_2) b \cdot s \right) \right)\\
  &=& \delta_{j,0} (a_1\cdot r) 
 (B^i_0 (a_2) b \cdot s)\\
 &=& \delta_{j,0} (a_1\cdot r) 
 ((-1)^i (\delta\circ\sigma^{-1})^i (a_2) b \cdot s)\\
 &=& \delta_{j,0}\delta_{i,0} (a_1\cdot r) 
 (a_2 b \cdot s),
\end{eqnarray*}
where, in the last equality, we use that $J \cdot R =0$.
On the other hand,
\begin{eqnarray*}
x^ja \triangleright (r(x^ib \triangleright s))&=&x^ja \triangleright ( \delta_{i,0} r (b \cdot s)) \\
& = & \delta_{j,0} \delta_{i,0} (a \cdot (r (b \cdot s)) \\
& = & \delta_{j,0} \delta_{i,0} (a_1 \cdot r)(a_2b \cdot s).
\end{eqnarray*}
Similarly, if the partial action $\cdot$ is symmetric, then $\triangleright$ is also symmetric.
\end{proof}

In particular, when $\delta = 0$, the above result ensures that always is possible to extend any partial action of $A$ on $R$ to  the Hopf-Ore extension $A[x,\sigma]:= A[x,\sigma, 0]$, in the trivial way.

\subsection{Characterization of partial actions of Hopf-Ore extensions}

Let $\cdot : \mathcal{H}  \otimes R \to R$ a partial action and $g \in G(\mathcal{H})$.
The element $g \cdot 1_R \in R$ is always an idempotent element, and $R$ has at least two: $0$ and $1_R$.
 As consequence of Lemma \ref{parte_global}, any partial action of $A[x,\sigma, \delta]$ on $R$ such that $g \cdot 1_R = 1_R$ is totally characterized by the partial action of $A$ on $R$ and a compatible skew-derivation determined by $x$.
 
 We investigate now the case when $g \cdot 1_R = 0$. In particular, in this case, we obtain that the partial action of $A[x,\sigma]$ on $R$ depends only on the partial action of $A$ on $R$, a skew-derivation (determined by the particular element $x \cdot 1_R$), and a compatibility relation between them.
  
  The following result shows this fact in a more general way.
\begin{thm}\label{formula_eh_acao_parcial_geral}
    Let $\mathcal{H}$ be a Hopf algebra, $\cdot : \mathcal{H} \o R \to R$ a partial action, $g \in G(\mathcal{H})$ and $x \in P_{1,g}(\mathcal{H})$ such that $gx = q^{-1}xg$, for some $q\in \Bbbk^{\times}$. 
	If $g \cdot 1_R = 0$, then
\begin{align}\label{formula}
	x^ja \cdot r =  \sum_{k=0}^j (-1)^k q^{ -\frac{k(k-1)}{2}} { j \choose k }_{q^{-1}} (x \cdot 1_R)^{j-k}(g^{k}a\cdot r) (x \cdot 1_R)^k,
	\end{align}
	for any $a \in \mathcal{H}, r \in R$ and $j \in \N_0$.
\end{thm}
 \begin{proof}
	First, by Lemma \ref{acao_produto_de_skew_primitivos_nichols} we know that $g^{-1}x \cdot r = -q r (x \cdot 1_R)$.
Now, we prove the desired equality by induction on $j$. It is immediate for $j=0$, and if one wants to check for $j=1$ it follows by condition (PA.3) for the expression $g^{-1}x \cdot (ga \cdot r)$ that $xa\cdot r=(x\cdot 1_R)(a\cdot r)-(ga\cdot r)(x\cdot 1_R).$

Assume $j \geq 0$.
Then, by (PA.3), we get
	\begin{align*}
	g^{-1}x \cdot (gx^ja \cdot r) = (g^{-1}x \cdot 1_R)(g^{-1}gx^ja \cdot r)+ (1 \cdot 1_R)(g^{-1}xgx^{j}a \cdot r),
	\end{align*}
	that is, $-q (gx^ja \cdot r)(x \cdot 1_R) = -q (x \cdot 1_R)(x^ja \cdot r) + q(x^{j+1}a \cdot r),$
	and thus $x^{j+1}a \cdot r  = (x \cdot 1_R)(x^ja \cdot r)-( gx^ja \cdot r)(x \cdot 1_R).$

	Now, by the induction hypothesis, we obtain
	\begin{eqnarray*}
	& & x^{j+1}a \cdot r \\
	& = & (x \cdot 1_R) \left( \sum_{k=0}^j (-1)^k q^{- \frac{k(k-1)}{2}} { j \choose k }_{q^{-1}} (x \cdot 1_R)^{j-k}(g^{k}a \cdot r)(x \cdot 1_R)^k \right) \\
	& & \!   -  \left(\!\! q^{-j} \sum_{k=0}^j (-1)^k q^{- \frac{k(k-1)}{2}} { j \choose k }_{q^{-1}} \! \!\! \! \! (x \cdot 1_R)^{j-k}(g^{k+1}a \cdot r)(x \cdot 1_R)^k \!\!\right)\!\!(x \cdot 1_R) \\
	& = &  (x \cdot 1_R)^{j+1}(a \cdot r) \\
	& & + \sum_{k=1}^j (-1)^k q^{- \frac{k(k-1)}{2}} { j \choose k }_{q^{-1}} (x \cdot 1_R)^{j+1-k}(g^{k}a \cdot r)(x \cdot 1_R)^k \\
	& & - q^{-j} \sum_{k=0}^{j-1} (-1)^k q^{- \frac{k(k-1)}{2}} { j \choose k }_{q^{-1}} (x \cdot 1_R)^{j-k}(g^{k+1}a \cdot r)(x \cdot 1_R)^{k+1} \\
	& & - q^{-j} (-1)^j q^{-\frac{j(j-1)}{2}} (g^{j+1}a \cdot r) (x \cdot 1_R)^{j+1}  \\
	& = &  (x \cdot 1_R)^{j+1}(a \cdot r) \\
	& & + \sum_{k=1}^j (-1)^k q^{- \frac{k(k-1)}{2}} \left({ j \choose k }_{q^{-1}} \!\!\!\!\!+ q^{-j+(k-1)} { j \choose k-1 }_{q^{-1}} \right)(x \cdot 1_R)^{j+1-k} \\
&	& \times  (g^{k}a \cdot r)(x \cdot 1_R)^k + (-1)^{j+1} q^{-\frac{j(j+1)}{2}} (g^{j+1}a \cdot r) (x \cdot 1_R)^{j+1}  \\
	& \stackrel{\eqref{6.7}}{=} &  (x \cdot 1_R)^{j+1}(a \cdot r) \\
&	& +  \sum_{k=1}^j (-1)^k q^{- \frac{k(k-1)}{2}} { j + 1 \choose k }_{q^{-1}} (x \cdot 1_R)^{j+1-k}(g^{k}a \cdot r)(x \cdot 1_R)^k \\
&	& +  (-1)^{j+1} q^{-\frac{j(j+1)}{2}} (g^{j+1}a \cdot r) (x \cdot 1_R)^{j+1} 
	\end{eqnarray*}
	that is,
	\begin{eqnarray*}
	x^{j+1}a \cdot r 
	& = &  \sum_{k=0}^{j+1} (-1)^k q^{- \frac{k(k-1)}{2}} { j + 1 \choose k }_{q^{-1}} (x \cdot 1_R)^{j+1-k}(g^{k}a \cdot r)(x \cdot 1_R)^k.
	\end{eqnarray*}
	It concludes the induction step, and so expression \eqref{formula} holds.
\end{proof}

\begin{cor}\label{formula_eh_acao_parcial}
    If $\mathcal{H}=A[x,\sigma]$, then Theorem \ref{formula_eh_acao_parcial_geral} characterizes via \eqref{formula} all the partial actions of $A[x,\sigma]$ on $R$ such that $g \cdot 1_R=0$.
\end{cor}

The next result gives us a necessary condition for some partial actions of $A[x,\sigma]$ on $R$.
In \S \ref{Sec_Examples}, we see that such condition is sufficient for certain cases.

\begin{cor}\label{cor_ida}
       Let $\cdot : A[x,\sigma] \otimes R \to R$ be a partial action such that $g \cdot 1_R =0$.
    Then $a \cdot (x \cdot r) =(a_1 \cdot 1_R)(x \cdot 1_R) (\sigma^{-1}(a_2) \cdot r)$, for all $a \in A, r \in R$.
\end{cor}

\begin{proof}
    Since
    \begin{eqnarray*}
 & & (g \sigma^{-1}(a_1) g^{-1} \cdot 1_R) (g a_2 \cdot 1_R) \\
 & = & (g \sigma^{-1}(a_1) g^{-1} \cdot 1_R) (g a_2 g^{-1}g \cdot 1_R) \\
 & \stackrel{\ref{remark_cond_hopf ore}}{=} & ((g \sigma^{-1}(a)g^{-1})_1 \cdot 1_R))((g \sigma^{-1}(a) g^{-1})_2 g \cdot 1_R) \\
& = & g \sigma^{-1}(a) g^{-1} \cdot 
 ( g \cdot 1_R) \\
& = & 0,
    \end{eqnarray*}
   we get 
    \begin{eqnarray*}
         && a \cdot (x \cdot r) \\
         &=& a \cdot ((x \cdot 1_R)r)\\
       & = & \left(a_1 \cdot (x \cdot 1_R)\right)(a_2 \cdot r) \\
        & = & (a_1 \cdot 1_R) (a_2x \cdot 1_R)(a_3 \cdot r) \\
        & = & (a_1 \cdot 1_R) (x\sigma^{-1}(a_2) \cdot 1_R)(a_3 \cdot r)\\
        & \stackrel{\ref{remark_cond_hopf ore}}{=}  & (g \sigma^{-1}(a_1) g^{-1} \cdot 1_R) (x a_2 \cdot 1_R)(a_3 \cdot r) \\
          & \stackrel{(\ref{formula})}{=} & (g \sigma^{-1}(a_1) g^{-1} \cdot 1_R) \left((x \cdot 1_R) (a_2 \cdot 1_R) - (g a_2 \cdot 1_R)(x\cdot 1_R)\right)(a_3 \cdot r) \\
          & =& (g \sigma^{-1}(a_1) g^{-1} \cdot 1_R) (x \cdot 1_R) (a_2 \cdot 1_R)(a_3 \cdot r) \\
          & & - (g \sigma^{-1}(a_1) g^{-1} \cdot 1_R) (g a_2 \cdot 1_R)(x\cdot 1_R)(a_3 \cdot r)\\
           &=& (g \sigma^{-1}(a_1) g^{-1} \cdot 1_R) (x \cdot 1_R)( a_2 \cdot 1_R) (a_3 \cdot r)
          \\
          & \stackrel{\ref{remark_cond_hopf ore}}{=} & (a_1 \cdot 1_R)(x \cdot 1_R) (\sigma^{-1}(a_2) \cdot 1_R) (a_3 \cdot r)
          \\
           &=& (a_1 \cdot 1_R)(x \cdot 1_R) (\sigma^{-1}(a_2) \cdot r).
    \end{eqnarray*}
    Therefore, $a \cdot (x \cdot r) =(a_1 \cdot 1_R)(x \cdot 1_R) (\sigma^{-1}(a_2) \cdot r)$.
\end{proof}

\medbreak

In some situations the map \eqref{formula} can be improved due to $q$-combinatoric properties, as in the next result.

\begin{lem}
     \label{lemma_truncamento}
Let $\cdot: A \otimes R \to R$ be a linear map, $w \in R$ and $q \in \Bbbk^\times$.
Consider a Hopf-Ore extension of $A$ and the linear map $\triangleright : A [x, \sigma] \o R \to R$ defined as 
\begin{align*}
x^ja \triangleright r =  \sum_{k=0}^j (-1)^k q^{ -\frac{k(k-1)}{2}} { j \choose k }_{q^{-1}} w^{j-k}(g^{k}a\cdot r) w^k.
\end{align*}
If $q \in \mathbb{G}'_M$, then for all $j \in \N_0, a\in A, r \in R$,
\begin{itemize}
    \item[(i)] $x^ja \triangleright r = \sum_{\ell = 0}^{j_D} (-1)^{\ell} \binom{j_D}{\ell}  w^{M(j_D-\ell)}\left(x^{j_R}g^{M\ell} a \triangleright r \right)w^{M \ell},$
\item[(ii)] $x^M \triangleright 1_R = \left(1_R - (g^M\cdot 1_R)\right) w^M,$
\item[(iii)] if $g^M =1$ and $w^M \in Z(R)$, then $x^ja \triangleright r = 0$, for $j\geq M$.
\end{itemize}
\end{lem}

\begin{proof}
Set $j = M j_D + j_R$ with $0\leq j_R < M$ and $a\in A, r \in R$.
For item (i), we have by \eqref{radford} that
\begin{eqnarray*}
x^ja \triangleright r &=& \sum_{k=0}^j  (-1)^k q^{-\frac{k(k-1)}{2}} { j \choose k }_{q^{-1}}  w^{j-k} (g^{k}a \cdot r )w^{k} \\
&=& \sum_{\ell=0}^{j_D}\sum_{k=0}^{j_R}  (-1)^{M \ell + k} q^{-\frac{(M \ell + k)(M \ell + k-1)}{2}}  { j_D \choose \ell } { j_R \choose k }_{q^{-1}} \\
& & \times w^{M j_D + j_R - (M\ell + k)} (g^{M \ell + k}a \cdot r) w^{M \ell + k} \\
&=& \sum_{\ell=0}^{j_D}\sum_{k=0}^{j_R}   (-1)^{\ell + k} q^{-\frac{k (k-1)}{2}} { j_D \choose \ell } { j_R \choose k }_{q^{-1}}  \\
& & \times w^{M (j_D - \ell) + j_R -  k} (g^{M \ell + k}a \cdot r) w^{M \ell + k} \\
&=& \sum_{\ell=0}^{j_D} 
 (-1)^{\ell} { j_D \choose \ell } w^{M(j_D-\ell)} \\
 & & \times \left( \sum_{k=0}^{j_R}    (-1)^{k} q^{-\frac{k (k-1)}{2}} { j_R \choose k }_{q^{-1}} w^{j_R - k} (g^k g^{M \ell}a \cdot r) w^{k}\right) w^{M \ell} \\
 &=&  \sum_{\ell=0}^{j_D} 
  (-1)^{\ell} { j_D \choose \ell } w^{M(j_D-\ell)} 
 \left(x^{j_R}g^{M\ell} a \triangleright r\right) w^{M\ell}.
\end{eqnarray*}
Item (ii) is immediate. For item (iii), assume $g^M =1$ and $w^M \in Z(R)$.
Then, for $j\geq M$,
\begin{eqnarray*}
x^ja \triangleright r &=& \sum_{\ell=0}^{j_D} 
  (-1)^{\ell} { j_D \choose \ell } w^{M(j_D-\ell)} 
 \left(x^{j_R}g^{M\ell} a \triangleright r\right) w^{M\ell} \\
 &= &\sum_{\ell=0}^{j_D} 
  (-1)^{\ell} { j_D \choose \ell } w^{M j_D} 
 \left(x^{j_R} a \triangleright r\right) \\
&=& w^{M j_D} (x^{j_R} a \triangleright r)  \sum_{\ell=0}^{j_D}   (-1)^{\ell}
{ j_D \choose \ell } \\
&=& 0.
\end{eqnarray*}
\end{proof}

\subsection{Sufficient conditions for extending a partial action to a Hopf-Ore extension}

In Theorem \ref{formula_eh_acao_parcial_geral}, we have an explicit expression for a partial action of a Hopf-Ore extension $A [x, \sigma]$ on $R$ when $g \cdot 1_R = 0$. 
Now, we investigate when a  partial action of $A$ on $R$, such that  $g \cdot 1_R = 0$, can be extended to a partial action of $A [x, \sigma]$ on $R$ through \eqref{formula}.

Observe that, from now on, we use the notation $\cdot$ for two different linear maps.
It does not create an ambiguity issue since the second map extends the first one.

\begin{thm}\label{formula_eh_acao_parcial_volta}
	Consider $\cdot : A \o R \to R$ a partial action and $A [x, \sigma]$ a Hopf-Ore extension.
 If $g \cdot 1_R = 0$, we define the linear map $\cdot : A [x, \sigma] \o R \to R$ by
\begin{align}\label{formula_ext}
	x^ja \cdot r =  \sum_{k=0}^j (-1)^k q^{ -\frac{k(k-1)}{2}} { j \choose k }_{q^{-1}} w^{j-k}(g^{k}a\cdot r) w^k,
	\end{align}
	for $q = \chi (g)$ and some fixed element $w\in R$. 
 Then, for all $a$, $b\in A$, $i, j \in \N_0$ and $r,s\in R$,
        \begin{itemize}
        \item[(i)] $1_H\cdot r=r$,
        \item[(ii)] $x^ja\cdot (rs)=((x^ja)_1\cdot r)((x^ja)_2\cdot s)$,
        \item[(iii)] if
         \begin{equation}\label{cod_volta}
      a\cdot (x^jb\cdot r)=(a_1\cdot 1_R)(a_2x^jb\cdot r),
  \end{equation}
then $x^ja\cdot(x^ib\cdot r)=((x^ja)_1\cdot 1_R)((x^ja)_2x^ib\cdot r),$ and
         \item[(iv)] if
 \begin{equation}\label{cod_voltasym}
   a\cdot (x^jb\cdot r)=(a_1x^jb\cdot r)(a_2\cdot 1_R),
  \end{equation}
then $x^ja\cdot(x^ib\cdot r)=((x^ja)_1x^ib\cdot r)((x^ja)_2\cdot 1_R).$
 \end{itemize}
 \end{thm}
Therefore, if \eqref{cod_volta} holds, the linear map $\cdot$ is a partial action of $A [x, \sigma]$ on $R$. Moreover, it is symmetric if \eqref{cod_voltasym} is also satisfied.

\begin{proof} Clearly item $(i)$ holds by definition.

$(ii)$ Note that, as $\cdot$ is a partial action of $A$ on $R$,
\begin{eqnarray*}
x^ja \cdot (rs) &=& \sum_{k=0}^j (-1)^{k} q^{-\frac{k(k-1)}{2}} {j \choose k}_{q^{-1}} w^{j-k}  (g^{k}a_1\cdot r)(g^{k}a_2\cdot s) w^{k}.
\end{eqnarray*}
Thus, for $ k=k_0 \in\I_{0, j}$ fixed, we obtain the following element  
\begin{equation}\label{eq_vetor}
\left( (-1)^{k_0} q^{- \frac{k_0(k_0-1)}{2}} {j \choose k_0}_{q^{-1}} \right) \ w^{j-k_0} (g^{k_0}a_1 \cdot r) (g^{k_0}a_2 \cdot s) w^{k_0}.
\end{equation} 
On the other hand, 
\begin{eqnarray*}
& & ((x^ja)_1 \cdot r)((x^ja)_2 \cdot s) 
\\
& = &   \sum_{\ell =0}^j {j \choose \ell}_q  (g^{\ell}x^{j-\ell}a_1 \cdot r) ( x^\ell a_2 \cdot s) \\
&\stackrel{\ref{id_q}}{=} &   \sum_{\ell=0}^j q^{\ell(j-\ell)} {j \choose \ell}_{q^{-1}} q^{-\ell(j-\ell)} (x^{j-\ell} g^{\ell}a_1 \cdot r) ( x^{\ell}a_2 \cdot s)\\
&= & \sum_{\ell=0}^j  \sum_{r=0}^{j-\ell} \sum_{\theta=0}^{\ell} (-1)^{r+\theta} q^{-\frac{r(r-1)}{2}} q^{-\frac{\theta(\theta-1)}{2}} {j \choose \ell}_{q^{-1}} \!\!{j-\ell \choose r}_{q^{-1}} \!\!{\ell \choose \theta}_{q^{-1}}\\ 
& & \times   w^{j-(\ell+r)}  (g^{\ell+r}a_1\cdot r) w^{r+\ell-\theta} (g^{\theta}a_2\cdot s)w^{\theta}\\
& \stackrel{k=\ell+r}{=} & \sum_{\ell=0}^j  \sum_{k=\ell}^{j} \sum_{\theta=0}^{\ell}(-1)^{k-\ell+\theta} q^{-\frac{(k-\ell)(k-\ell-1)}{2}}\! q^{-\frac{\theta(\theta-1)}{2}} \!{j \choose \ell}_{q^{-1}} \!\!{j-\ell \choose k-\ell}_{q^{-1}} \!\!{\ell \choose \theta}_{q^{-1}} \\
& & \times  w^{j-k}  (g^{k}a_1\cdot r) w^{k-\theta} (g^{\theta}a_2\cdot s)w^{\theta}.
\end{eqnarray*}
Thus, using \eqref{Lema2_2} and for $ k_0 \in\I_{0, j}$ fixed, we have that each $ \ell \in\I_{0, j}$ determines the multiple of the fixed element \eqref{eq_vetor}(considering $k=k_0$ and $\theta=k_0$),
\begin{eqnarray*}  & & \sum_{\ell=0}^{j}  (-1)^{\ell} q^{-\frac{k_0(k_0-1)}{2}}q^{\ell k_0}q^{-\frac{\ell(\ell+1)}{2}}q^{-\frac{k_0(k_0-1)}{2}} {j \choose k_0}_{q^{-1}} {k_0 \choose \ell}_{q^{-1}} {\ell \choose k_0}_{q^{-1}} \\
&& \times     w^{j-k_0} (g^{k_0}a_1 \cdot r) (g^{k_0}a_2 \cdot s) w^{k_0}.
\end{eqnarray*}
Therefore, looking to \eqref{eq_vetor} and as $ k_0 \in\I_{0, j}$, we just need to see that
\begin{eqnarray*}	
\sum_{\ell=0}^{k_0}  (-1)^{\ell-k_0} q^{-\frac{k_0(k_0-1)}{2}} q^{\ell k_0}q^{-\frac{\ell(\ell+1)}{2}}  {k_0 \choose \ell}_{q^{-1}} {\ell \choose k_0}_{q^{-1}}&=&1,
\end{eqnarray*}
but it holds by \eqref{lema_acao_da_acao}.

$(iii)$ Suppose $a\cdot (x^jb\cdot r)=(a_1\cdot 1_R)(a_2x^jb\cdot r)$. Then,
\begin{eqnarray*}
& & x^ja\cdot(x^ib\cdot r) \\
&= & \sum_{k=0}^j (-1)^{k} q^{-\frac{k(k-1)}{2}} {j \choose k}_{q^{-1}} w^{j-k}  (g^{k}a_1\cdot 1_R)(g^{k}a_2x^ib\cdot r) w^{k}\\
& = &  \sum_{k=0}^j (-1)^{k} q^{-\frac{k(k-1)}{2}} {j \choose k}_{q^{-1}} w^{j-k}  (g^{k}a_1\cdot 1_R)(x^i\sigma^{-i}(g^{k}a_2)b\cdot r) w^{k}\\
& =  &  \sum_{k=0}^j \sum_{\ell=0}^i (-1)^{k+\ell} q^{-\frac{k(k-1)}{2}}q^{-\frac{\ell(\ell-1)}{2}} {j \choose k}_{q^{-1}}{i \choose \ell}_{q^{-1}} w^{j-k}  (g^{k}a_1\cdot 1_R) w^{i-\ell} \\
& & \times \ (g^{\ell}\sigma^{-i}(g^{k}a_2)b\cdot r) w^{\ell+k}.
\end{eqnarray*}
Thus, fixing $ k=k_0 \in\I_{0, j} $ and $\ell=\ell_{0}\in\I_{0, i}$, we have

\begin{align}\label{eq_vetor2}
\begin{split}
& (-1)^{k_0+\ell_0} q^{-\frac{k_0(k_0-1)}{2}}q^{-\frac{\ell_0(\ell_0-1)}{2}} {j \choose k_0}_{q^{-1}}{i \choose \ell_0}_{q^{-1}} \\
\times &  w^{j-k_0}  (g^{k_0}a_1\cdot 1_R)w^{i-\ell_0} (g^{\ell_0}\sigma^{-i}(g^{k_0}a_2)b\cdot r) w^{\ell_0+k_0}.
\end{split}
\end{align}

On the other hand, using the same computations of $(ii)$,
\begin{eqnarray*}
&& ((x^ja)_1\cdot 1_R)((x^ja)_2x^ib\cdot r) \\
& =&  \sum_{\ell=0}^j {j \choose \ell}_{q^{-1}} (x^{j-\ell} g^{\ell}a_1 \cdot 1_R) ( x^{i+\ell}\sigma^{-i}(a_2)b \cdot r)\\
&= & \sum_{\ell=0}^j  \sum_{k=\ell}^{j} \sum_{\theta=0}^{i+\ell}(-1)^{k-\ell+\theta} q^{-\frac{(k-\ell)(k-\ell-1)}{2}} q^{-\frac{\theta(\theta-1)}{2}} \!{j \choose \ell}_{q^{-1}}\!\! {j-\ell \choose k-\ell}_{q^{-1}} \!\!{i+\ell \choose \theta}_{q^{-1}} \\
&& \times  \ w^{j-k}  (g^{k}a_1\cdot 1_R) w^{k+i-\theta} (g^{\theta}\sigma^{-i}(a_2)b\cdot r)w^{\theta}.
\end{eqnarray*}
Then, using \eqref{Lema2_2} and for $ k_0 \in\I_{0, j}$ fixed, we have that each $ \ell \in\I_{0, j}$ we can choose $k=k_0$ and $\theta=k_0+\ell_0$, obtaining the element
\begin{eqnarray*} && \!\!\!\! \sum_{\ell=0}^{j}  (-1)^{\ell_0-\ell}  q^{-\frac{k_0(k_0-1)}{2}}q^{\ell k_0-k_0\ell_0}q^{-\frac{\ell(\ell+1)}{2}}q^{-\frac{k_0(k_0-1)}{2}} q^{-\frac{\ell_0(\ell_0-1)}{2}}  \!{j \choose k_0}_{q^{-1}} \!\!{k_0 \choose \ell}_{q^{-1}} \\
&& \times  \ {i+\ell \choose k_0+\ell_0}_{q^{-1}} q^{ik_0}  w^{j-k_0} (g^{k_0}a_1 \cdot 1_R) w^{i-\ell_0} (g^{\ell_0}\sigma^{-i}(g^{k_0}a_2)b\cdot r) w^{k_0+\ell_0}.
\end{eqnarray*}

Therefore, by \eqref{eq_vetor2}, we just need to see that
\begin{eqnarray*} \sum_{\ell=0}^{k_0}  (-1)^{\ell+k_0} q^{\ell k_0-k_0\ell_0}q^{-\frac{\ell(\ell+1)}{2}}q^{-\frac{k_0(k_0-1)}{2}} q^{ik_0} \!{k_0 \choose \ell}_{q^{-1}} \!\!{i+\ell \choose k_0+\ell_0}_{q^{-1}}= {i \choose \ell_0}_{q^{-1}}, 
\end{eqnarray*}

which is true again for \eqref{lema_acao_da_acao}.

$(iv)$ Analogously to item (iii).
\end{proof}

\begin{rem}\label{remark_cod}
     Notice that $w=x\cdot 1_R$ within Theorem \ref{formula_eh_acao_parcial_volta}. Moreover, using \eqref{formula_ext}, equalities \eqref{cod_volta} and  \eqref{cod_voltasym} can be rewritten, respectively, as
    \begin{align}
        \begin{split}\label{somatorio_acao}
     & \sum_{k=0}^j (-1)^k q^{ -\frac{k(k-1)}{2}} { j \choose k }_{q^{-1}} (a_1\cdot w^{j-k})(a_2g^{k}b\cdot r)(a_3\cdot w^k)\\
    = & \sum_{k=0}^j (-1)^k q^{ -\frac{k(k-1)}{2}} { j \choose k }_{q^{-1}} (a_1\cdot 1_R) w^{j-k}(g^{k}\sigma^{-j}(a_2)b\cdot r)w^k
    \end{split} 
         \end{align}
and
      \begin{align}
\begin{split}\label{somatorio_acao_simetrica}
     &\sum_{k=0}^j (-1)^k q^{ -\frac{k(k-1)}{2}} { j \choose k }_{q^{-1}} (a_1\cdot w^{j-k})(a_2g^{k}b\cdot r)(a_3\cdot w^k)\\
     = & \sum_{k=0}^j (-1)^k q^{ -\frac{k(k-1)}{2}} { j \choose k }_{q^{-1}} w^{j-k}(g^{k}\sigma^{-j}(a_1)b\cdot r)w^k (a_2\cdot 1_R).
     \end{split}
    \end{align} 
\end{rem}

\medbreak

In order to prove \eqref{somatorio_acao}, obviously a sufficient condition is
\begin{align}\label{igualdade_pontual}
(a_1\cdot w^{j-k})(a_2g^{k}b\cdot r)(a_3\cdot w^k) = (a_1\cdot 1_R) w^{j-k}(g^{k}\sigma^{-j}(a_2)b\cdot r)w^k
\end{align}
for all $a, b \in A , r\in R$ and $0 \leq k \leq j$.
But, we emphasize that the above condition is not necessary, \emph{i.e.}, equality \eqref{somatorio_acao} can hold even if \eqref{igualdade_pontual} does not.

For instance, consider Sweedler's Hopf algebra $\mathbb{H}_4=\langle g, x \ | \ g^2 =1, x^2 =0, xg=-gx \rangle$, where $g \in G(\mathbb{H}_4)$ and $x \in P_{1,g}(\mathbb{H}_4)$, Hopf-Ore extension $\mathbb{H}_4[y,\sigma]$, where $\sigma(g)=-g$ and $\sigma(x)=-x$, and the partial action of $\mathbb{H}_4$ on an algebra $R$ given by $1 \cdot r=r$, $g \cdot r=0$, $x \cdot r = \Omega r$ and $g x \cdot r = r \Omega$, for some $\Omega \in R$ such that $\Omega^2 \in Z(R)$ (see \cite[Ex. 3.13]{FMS2}).
If $y \cdot 1_R=w$, then for $j=1, a=gx$ and $b=g$, on the one hand
$$
 (a_1\cdot w^{1-k})(a_2g^{k}b\cdot r)(a_3\cdot w^k)= w^{1-k} (g^{k+1} \cdot r) w^k \Omega 
$$
and, on the other hand
$$
(a_1\cdot 1_R) w^{1-k}(g^{k}\sigma^{-1}(a_2)b\cdot r)w^k= - \Omega w^{1-k} (g^{k} \cdot r) w^k - w^{1-k} (g^{k}x \cdot r) w^k. 
$$
Thus, for \eqref{igualdade_pontual} to be true for $k=0$ and $k=1$, one needs that $\Omega w + w \Omega = 0$. 
But, for equality \eqref{somatorio_acao} to be valid, one just needs that $\Omega w + w \Omega \in Z(R)$. 

In \S \ref{Sec_Examples}, we see that condition $\Omega w + w \Omega \in Z(R)$ actually is the necessary and sufficient condition to extend the partial action of $\mathbb{H}_4$ on $R$ to and partial action of $\mathbb{H}_4[y,\sigma]$ on $R$.
Furthermore, we prove that \eqref{somatorio_acao} and \eqref{igualdade_pontual} are equivalents when $A=\Bbbk G$; the same happens for a general $A$
if $g \in Z(A)$ and $w \in Z(R)$.

\subsection{Partial actions of Hopf-Ore extensions factorized through quotients} \label{Sec_quotient}

Let $\mathcal{H}$ be a Hopf algebra and $I$ a Hopf ideal of $\mathcal{H}$.
In general, we have that if $\cdot : \mathcal{H} \o R \to R$ is a (symmetric) partial action such that $I\cdot R=0$, then the map $\cdot$ induces a (symmetric) partial action of $\mathcal{H}/I$ on $R$.
Conversely, if there is a (symmetric) partial action of $\mathcal{H}/I$ on $R$, then  this action is induced from one of $\mathcal{H}$ on $R$.

In the previous section we determine conditions for extending a partial action of a Hopf algebra $A$ on $R$ to a partial action $\cdot: A[x,\sigma, \delta]\o R\to R$.
Then, if there is a Hopf ideal $I$ of $A[x,\sigma, \delta]$ such that $I \cdot R=0$, we produce partial actions of $A[x,\sigma, \delta]/I$ on $R$ from partial actions of $A$ on $R$. 
Next, we investigate when such procedure is exhaustive.

Let $I$ be a Hopf ideal of $A[x,\sigma, \delta]$ and $\triangleright:A[x,\sigma,\delta]/I \o R \to R$ a (symmetric) partial action.
Then $\triangleright$ defines  two (symmetric) partial actions $\cdot: A\otimes R\to R $ and
$\bullet: A[x,\sigma,\delta] \o R \to R$, via $a\cdot r=\overline{a}\triangleright r$ and $x^ja\bullet r=\overline{x^ja}\triangleright r$, respectively.

If $\overline{g} \triangleright 1_R = 1_R$, by Proposition \ref{parte_global}, the (symmetric) partial action $\bullet$ coincides with the extension of the (symmetric) partial action $\cdot$ to the Hopf Ore extension. Now, if $\overline{g} \triangleright 1_R = 0$ and $\delta=0$, the (symmetric) partial action $\bullet$ coincides with the extension of the (symmetric) partial action $\cdot$ to the Hopf Ore extension, putting $w=\overline{x} \triangleright 1_R$, since that the map $\cdot$ satisfies the equalities \eqref{somatorio_acao} (and \eqref{somatorio_acao_simetrica}).

Thus, in both cases $I\bullet R=0$, what induces a (symmetric) partial action of $A[x,\sigma,\delta]/I$, which turns out to be the original one $\triangleright$, \emph{i.e.}, for these cases such procedure is exhaustive.

\medbreak

To finish this section, we determine sufficient conditions to induce a partial action of $A[x, \sigma]$ to a partial action of the quotient $A[x, \sigma]/I$, for two specific families of Hopf ideals $I$.
Such results allow us to compute in \S \ref{Sec_Rank_one} some (symmetric)  partial actions of the rank one Hopf algebras.

\begin{prop} \label{prop:rank_one_nilp}
Let $A[x, \sigma]$ be a Hopf-Ore extension  and $q = \chi(g) \in \mathbb{G}'_d$.
If $\cdot: A[x, \sigma] \otimes R\to R$ is a partial action such that:
\begin{itemize}
    \item[(i)] $g\cdot 1_R = 1_R$ and the map $x\cdot \underline{ \ \ }  $ is $d$-nilpotent, or
    \item[(ii)] $g\cdot 1_R = 0$ and  $ x\cdot 1_R = 0$, or
    \item[(iii)] $g\cdot 1_R = 0, g^d\cdot 1_R = 1_R$ and $(x\cdot 1_R)^d \, r = (g^d\cdot r)(x\cdot 1_R)^d $, for all $r\in R$,
\end{itemize}
then $\cdot{|_{I\otimes R}} =0$, where $I$  is the (Hopf) ideal $ \langle x^d \rangle$.
In particular, the partial action $\cdot$ induces a partial action of $A[x, \sigma]/I$ on $R$.
\end{prop}

\begin{proof}
First assume the conditions in $(i)$.
By Lemma \ref{parte_global}, $x$ acts globally.
Hence, for $j\geq d$ and $a \in A$, we have
\begin{align*}
    x^ja\cdot r = x^{j-d}\cdot (x^{d}\cdot (a \cdot r)) = 0,
\end{align*}
since the action by $x$ is $d$-nilpotent.

Now assume that $g \cdot 1_R=0$. 
Corollary \ref{formula_eh_acao_parcial} shows us that \eqref{formula} holds.
Then, by Lemma \ref{lemma_truncamento}, we get that
\begin{align*}
x^ja \cdot r &= \sum_{\ell = 0}^{j_D} (-1)^{\ell} \binom{j_D}{\ell}  (x \cdot 1_R)^{d(j_D-\ell)}\left(x^{j_R}g^{d\ell} a \cdot r \right)(x \cdot 1_R)^{d \ell},
\end{align*}
for all $j \in \N_0, a\in A, r \in R$.

If $x\cdot 1_R = 0 $, clearly $x^ja \cdot r = 0 $, for $j\geq d$.

At last, assume conditions in $(iii)$.  
Thus, $g^{d\ell}\cdot 1_R = 1_R$, for $\ell\geq 0$, what implies that
$g^{d\ell}\cdot (y \cdot r) = g^{d\ell} y \cdot r$ for all $y\in A[x, \sigma], r \in R$.
Furthermore, from $(x\cdot 1_R)^d \, r = (g^d\cdot r)(x\cdot 1_R)^d $, it follows that $(x\cdot 1_R)^{d\ell} \, r = (g^{d\ell}\cdot r)(x\cdot 1_R)^{d\ell} $.
Hence,
\begin{align*}
x^ja \cdot r &= \sum_{\ell = 0}^{j_D} (-1)^{\ell} \binom{j_D}{\ell}  (x \cdot 1_R)^{d(j_D-\ell)}\left(x^{j_R}g^{d\ell}a  \cdot r \right)(x \cdot 1_R)^{d \ell} \\
&= \sum_{\ell = 0}^{j_D} (-1)^{\ell} \binom{j_D}{\ell} (x \cdot 1_R)^{d(j_D-\ell)}\left(g^{d\ell}x^{j_R}a  \cdot r \right)(x \cdot 1_R)^{d \ell} \\
&= \sum_{\ell = 0}^{j_D} (-1)^{\ell} \binom{j_D}{\ell} (x \cdot 1_R)^{d(j_D-\ell)}\left(g^{d\ell}\cdot (x^{j_R}a  \cdot r) \right)(x \cdot 1_R)^{d \ell} \\
&= \sum_{\ell = 0}^{j_D} (-1)^{\ell} \binom{j_D}{\ell} (x \cdot 1_R)^{d j_D}\left( x^{j_R}a  \cdot r \right) \\
&= (x \cdot 1_R)^{d j_D}\left( x^{j_R}a  \cdot r \right) \sum_{\ell = 0}^{j_D} \binom{j_D}{\ell} (-1)^{\ell} = 0,
\end{align*}
for all $j\geq d$ and $a \in A$.
Therefore, in any case, we  have $\cdot{|_{I\otimes R}} =0$ and so it induces a partial action of $A[x, \sigma]/I$ on $R$.
\end{proof}

\begin{prop}\label{prop:rank_one_non_nilp}
Let $A[x, \sigma]$ be a Hopf-Ore extension such that $q = \chi(g) \in \mathbb{G}'_d,$ $g^d \in Z(A)$ and $\chi^d = 1$.
If $\cdot: A[x, \sigma]\otimes R\to R$ is a partial action such that:
\begin{itemize}
    \item[(i)] $g\cdot 1_R = 1_R$ and  $(x^d +g^d)  \cdot r =  r $, for all $r\in R$, or 
    \item[(ii)] $g\cdot 1_R = 0$,  $g^d\cdot 1_R = 1_R$ and $[(x\cdot 1_R)^d -1_R] \, r = (g^d\cdot r)[(x\cdot 1_R)^d -1_R] $, for all $r\in R$, 
\end{itemize}
then $\cdot{|_{I\otimes R}} =0$, where $I$  is the (Hopf) ideal $ \langle x^d + g^d -1 \rangle$.
In particular, the partial action $\cdot$ induces a partial action of $A[x, \sigma]/I$ on $R$.
\end{prop}

\begin{proof}
First observe that $ x^j(x^d+g^d -1) a$, $ j \geq 0, a\in A$, generate linearly the ideal $I$ since $x^d, g^d \in Z(A[x, \sigma])$.

If $(i)$ holds, then Lemma \ref{parte_global} tells us that $g$ and $x$ act globally.
Thus,
\begin{eqnarray*}
x^j(x^d+g^d -1) a \cdot r = x^j \cdot ((x^d+g^d -1) \cdot (a \cdot r))  = 0, 
\end{eqnarray*}
for all $j\geq 0, a\in A, r\in R$.

On the other hand, assume the conditions in $(ii)$. Again we get from Corollary \ref{formula_eh_acao_parcial} that \eqref{formula} remains valid and consequently Lemma \ref{lemma_truncamento} gives us that
\begin{eqnarray*}
x^ja \cdot r &=& \sum_{\ell = 0}^{j_D} (-1)^{\ell} \binom{j_D}{\ell}  (x \cdot 1_R)^{d(j_D-\ell)}\left(x^{j_R}g^{d\ell}a  \cdot r \right)(x \cdot 1_R)^{d \ell},
\end{eqnarray*}
for all $j\geq 0, a\in A, r\in R$.
Hence, we have
\begin{eqnarray*}
&& x^j(x^d+g^d-1)a \cdot r \\
&=& x^{j+d}a \cdot r +x^{j}g^{d}a \cdot r - x^{j}a \cdot r \\
& = & \sum_{\ell = 0}^{j_D +1} (-1)^{\ell} \binom{j_D+1}{\ell}  (x \cdot 1_R)^{d(j_D + 1 -\ell)}\left(x^{j_R}g^{d\ell}a  \cdot r \right)(x \cdot 1_R)^{d \ell} \\
&&+ \sum_{\ell = 0}^{j_D} (-1)^{\ell}\binom{j_D}{\ell} (x \cdot 1_R)^{d(j_D-\ell)}\left(x^{j_R}g^{d(\ell +1)} a  \cdot r \right)(x \cdot 1_R)^{d \ell} \\
&&- \sum_{\ell = 0}^{j_D} (-1)^{\ell} \binom{j_D}{\ell}  (x \cdot 1_R)^{d(j_D-\ell)}\left(x^{j_R}g^{d\ell} a  \cdot r \right)(x \cdot 1_R)^{d \ell} \\
& =& \sum_{\ell = 0}^{j_D +1} (-1)^{\ell} \left(\binom{j_D}{\ell -1}+ \binom{j_D}{\ell}\right)  (x \cdot 1_R)^{d(j_D + 1 -\ell)}\left(x^{j_R}g^{d\ell} a  \cdot r \right)(x \cdot 1_R)^{d \ell} \\
&&+ \sum_{\ell = 0}^{j_D} (-1)^{\ell} \binom{j_D}{\ell}  (x \cdot 1_R)^{d(j_D-\ell)}\left[\left(x^{j_R}g^{d(\ell +1)} a -x^{j_R}g^{d\ell}a \right)  \cdot r \right](x \cdot 1_R)^{d \ell} \\
& =& \sum_{\ell = 0}^{j_D } (-1)^{\ell +1} \binom{j_D}{\ell } (x \cdot 1_R)^{d(j_D  -\ell)}\left(x^{j_R}g^{d(\ell+1)} a  \cdot r \right)(x \cdot 1_R)^{d (\ell+1)} \\
&& + \sum_{\ell = 0}^{j_D} (-1)^{\ell} \binom{j_D}{\ell}  (x \cdot 1_R)^{d(j_D + 1 -\ell)}\left(x^{j_R}g^{d\ell} a  \cdot r \right)(x \cdot 1_R)^{d \ell} \\
&&+ \sum_{\ell = 0}^{j_D} (-1)^{\ell} \binom{j_D}{\ell}  (x \cdot 1_R)^{d(j_D-\ell)}\left[x^{j_R}g^{d\ell} a (g^{d}-1)  \cdot r \right](x \cdot 1_R)^{d \ell} \\
&=& \sum_{\ell = 0}^{j_D} (-1)^{\ell} \binom{j_D}{\ell}  (x \cdot 1_R)^{d(j_D-\ell)}\left( \star \right)(x \cdot 1_R)^{d \ell},
\end{eqnarray*}
where $(\star) $ is
\begin{align*}
x^{j_R}g^{d\ell} a (g^{d}-1)  \cdot r  -(x^{j_R}g^{d(\ell+1)} a \cdot r)(x \cdot 1_R)^d + (x \cdot 1_R)^d (x^{j_R}g^{d\ell} a  \cdot r).
\end{align*}
But, by conditions $(ii)$, it follows that
\begin{eqnarray*}
&&(x \cdot 1_R)^d (x^{j_R}g^{d\ell} a  \cdot r )\\
&=& x^{j_R}g^{d\ell} a   \cdot r + ( g^d\cdot (x^{j_R}g^{d\ell} a  \cdot r )) ( (x \cdot 1_R)^d - 1_R) \\
&=& x^{j_R}g^{d\ell} a  \cdot r +  (g^d\cdot 1_R) (g^dx^{j_R}g^{d\ell} a  \cdot r ) ( (x \cdot 1_R)^d - 1_R) \\
&=& x^{j_R}g^{d\ell} a  \cdot r +   (x^{j_R}g^{d(\ell +1)} a  \cdot r ) ( (x \cdot 1_R)^d - 1_R),
\end{eqnarray*}
what implies that $x^j(x^d+g^d-1)a \cdot r = 0 $, for all $j\geq 0, a\in A, r\in R$.
Therefore, in any case, we  have $\cdot{|_{I\otimes R}} =0$ and so it induces a partial action of $A[x, \sigma]/I$ on $R$.
\end{proof}

\section{Applications and examples}\label{Sec_Examples}

Recall the discussion after Remark \ref{remark_cod}.
There, we obtain a sufficient condition to extend a partial action through a Hopf-Ore extension, but we pointed out that it may be not necessary.
So, in Sections  \ref{Sec_xcentral} and \ref{Sec_A=KG}, we investigate when we have sufficient and necessary conditions to Theorem \ref{formula_eh_acao_parcial_volta} characterizes some partial actions of a Hopf-Ore extension.

Furthermore, we investigate in \S \ref{Sec_Rank_one} the partial actions of rank one Hopf algebras (in particular, generalized Taft algebras and Radford algebras).
At last, we study the case for Hopf-Ore extensions of Sweedler's Hopf algebra and Nichols Hopf algebras, respectively in Sections \ref{Sec_Sweedler} and  \ref{Sec_Nichols}.

\subsection{The particular case when  the element $x\cdot 1_R \in Z(R)$} \label{Sec_xcentral}

We improve, under hypothesis, some results within \S \ref{sec:3}, as follows.

\begin{rem}\label{cond_ida2}
    Let $\cdot$ be a partial
    action of $A[x,\sigma]$ on $R$  
    such that $g \cdot 1_R =0$ and $x\cdot 1_R \in Z(R)$. If
    $ga\cdot 1_R=ag\cdot 1_R$, for $a\in A$, then $a \cdot (x \cdot 1_R) = (\sigma^{-1}(a) \cdot 1_R) (x \cdot 1_R).$
    Indeed, by Corollary \ref{cor_ida},
    \begin{eqnarray*}
        a \cdot (x \cdot 1_R) & = & (a_1 \cdot 1_R) (\sigma^{-1}(a_2) \cdot 1_R)(x \cdot 1_R)\\
     & =  & (g^{-1}a_1 g \cdot 1_R) (\sigma^{-1}(a_2) \cdot 1_R)(x \cdot 1_R)\\
     &  \stackrel{\ref{remark_cond_hopf ore}}{=}   & (\sigma^{-1}(a)_1 \cdot 1_R) (\sigma^{-1}(a)_2 \cdot 1_R)(x \cdot 1_R)\\
    &   = & (\sigma^{-1}(a) \cdot 1_R) (x \cdot 1_R).
    \end{eqnarray*}
    \end{rem}

  \begin{lem}\label{lema_volta2}
    Let $\cdot : A \otimes R \to R$ be a partial action, $g \in G(A)$ such that $g \cdot 1_R =0$, and $a \cdot w=(\sigma^{-1}(a)\cdot 1_R)w$, for  $a\in A$ and $w \in Z(R)$. Then, the following properties are satisfied:
    \begin{itemize}
        \item[(i)] $a\cdot w^j=(\sigma^{-j}(a)\cdot 1_R)w^j$,
        \item[(ii)] if $ga\cdot r=ag\cdot r$, then $(a_1\cdot w^{j})(a_2g^kb\cdot r)=(a_1\cdot 1_R)(g^k\sigma^{-j}(a_2)b\cdot r) w^j$,
        \item[(iii)] if $ga\cdot r=ag\cdot r$ and $\cdot$ is symmetric, then $(a_1g^kb\cdot r)(a_2\cdot w^{j})=(g^k\sigma^{-j}(a_1)b\cdot r)(a_2\cdot 1_R) w^j$,
    \end{itemize}
    for all $r\in R$, $j\in \N_{0}$ and $ k\in\mathbb{Z}$. 
  \end{lem}

   \begin{proof}
      $(i)$ We prove the desired equality by induction on $j$.
      It is immediate for $j=0$, and $j=1$ is valid by hypothesis. Now, suppose true for $j\geq 0$, then by  the induction hypothesis,
      \begin{eqnarray*}
          a\cdot w^{j+1} & = &  (a_1\cdot w)(a_2\cdot w^{j})\\
         & = &  (\sigma^{-1}(a_1)\cdot 1_R) (a_2\cdot w^{j})w\\
       & \stackrel{\ref{remark_cond_hopf ore}}{=}   &  (\sigma^{-1}(a)\cdot w^j)w\\
        & =  & (\sigma^{-j}(\sigma^{-1}(a))\cdot 1_R)w^{j+1}\\
        & = & (\sigma^{-(j+1)}(a)\cdot 1_R) w^{j+1}.
      \end{eqnarray*}

      $(ii)$ Using item (i) and $ga\cdot r=ag\cdot r$,
      \begin{eqnarray*}
          (a_1\cdot w^{j})(a_2g^kb\cdot r) & = & a\cdot (w^j(g^kb\cdot r))\\
         & = & a\cdot (w^j(bg^k\cdot r))\\
          & = & (a_1\cdot w^{j})(a_2bg^k\cdot r)\\
          & = & (\sigma^{-j}(a_1)\cdot 1_R)(a_2bg^k\cdot r) w^j \\
         & \stackrel{\ref{remark_cond_hopf ore}}{=} &  (a_1\cdot 1_R)(\sigma^{-j}(a_2)bg^k\cdot r) w^j\\
         & = & (a_1\cdot 1_R)(g^k\sigma^{-j}(a_2)b\cdot r) w^j.
  \end{eqnarray*}

  $(iii)$ Analogously,
      \begin{eqnarray*}
          (a_1g^kb\cdot r)(a_2\cdot w^{j}) & = & (a_1bg^k\cdot r)(\sigma^{-j}(a_2)\cdot 1_R)w^{j}\\
         & = & (g^k\sigma^{-j}(a_1)b\cdot r)(a_2\cdot 1_R) w^j. 
          \end{eqnarray*} 
    \end{proof}

Thus, we obtain the following characterization.

  \begin{thm}\label{teo_carac_extgrupo}
   Let $\cdot: A[x,\sigma] \otimes R \to R$ be a linear map  such that $x \cdot 1_R \in Z(R)$, $g\cdot 1_R=0$ and $ga\cdot r=ag\cdot r$, for all $a\in A$ and $r\in R$.
Then $\cdot$ is a (symmetric) partial action of $A[x,\sigma]$ on $R$, if and only if,  $\cdot |_{A\otimes R}$ is a (symmetric) partial action of $A$ on $R$, $a\cdot (x \cdot 1_R)=(\sigma^{-1}(a)\cdot 1_R)(x \cdot 1_R)$, for all $a\in A$, and $\cdot$ coincides with \eqref{formula_ext}, where $q=\chi(g)$ and $w= x\cdot 1_R$.
\end{thm}

\begin{proof}
    The first part follows directly of Theorem \ref{formula_eh_acao_parcial_geral} and Remark \ref{cond_ida2}, and the converse follows by results \ref{formula_eh_acao_parcial_volta}, \ref{remark_cod} and \ref{lema_volta2}.
\end{proof}

Notice that the condition $ga\cdot r=ag\cdot r$, in the above theorem, is not so restrictive, since a lot of Hopf algebras are quotients of Hopf-Ore extensions of group algebras, and thus $g$ is a central element in $A$ (see \cite[Prop. 2.2.]{Panov}).
In such case, the above result characterizes partial actions when $R$ is a domain, for instance.

\subsection{The particular case when $A= \Bbbk G$}\label{Sec_A=KG}

Similarly to \S \ref{Sec_xcentral}, we determine necessary and sufficient conditions to extend a (symmetric) partial action  of the group algebra $\Bbbk G$ on  $R$ to a (symmetric) partial action of $\Bbbk G[x,\sigma]$ on $R$.
In particular, in such case, the hypothesis $x \cdot 1_R \in Z(R)$ is not needed.

\begin{rem}\label{cond_ida}
     Let $\cdot$ be a partial
    action of $\Bbbk G[x,\sigma]$ on $R$  
    such that $g \cdot 1_R =0$. Then, $a \cdot (x \cdot 1_R) = (\sigma^{-1}(a) \cdot 1_R) (x \cdot 1_R)$,  for all $a\in  G$.
    Indeed, 
        \begin{eqnarray*}
        a \cdot (x \cdot 1_R) & \stackrel{\ref{cor_ida}}{=} & (a \cdot 1_R)(x \cdot 1_R) (\sigma^{-1}(a) \cdot 1_R)\\
        & = & \chi^{-1}(a) (a \cdot 1_R)(x \cdot 1_R) (a \cdot 1_R)\\
        & \stackrel{\eqref{ide_quasecentral}}{=} &\chi^{-1}(a) (a \cdot 1_R)(x \cdot 1_R)\\
        & = & (\sigma^{-1}(a) \cdot 1_R) (x \cdot 1_R).
    \end{eqnarray*}
       \end{rem}

  \begin{lem}\label{lema_volta}
    Let $\cdot : \Bbbk G \otimes R \to R$ be a partial action such that $g \cdot 1_R =0$ and $a \cdot w=(\sigma^{-1}(a)\cdot 1_R)w$, for $a\in G$ and $w\in R$. Then, for all $j\in \N$ and $ k\in\mathbb{Z}$, the following properties hold: 
    \begin{itemize}
        \item[(i)] $a\cdot w^j=(\sigma^{-j}(a)\cdot 1_R)w^j$,
        \item[(ii)] $(a\cdot w^{j-k})(ag^kb\cdot r)(a\cdot w^k)=(a\cdot 1_R)w^{j-k}(g^k\sigma^{-j}(a)b\cdot r) w^k$,
        \item[(iii)] if $\cdot$ is symmetric, then $$(a\cdot w^{j-k})(ag^kb\cdot r)(a\cdot w^k)=w^{j-k}(g^k\sigma^{-j}(a)b\cdot r) w^k(a\cdot 1_R).$$
    \end{itemize}
  \end{lem}
  \begin{proof} Analogous to the proof of Lemma \ref{lema_volta2}, but one needs to use Lemma \ref{ide_central} and $\sigma^{-1}(a)=\chi^{-1}(a)a$, for all $a\in G$.
    \end{proof}
 
    Therefore, we end up with the following result:
   \begin{thm}\label{teo_carac_grupo}
    Let $\cdot: \Bbbk G[x,\sigma]\otimes R \to R$ be a linear map  such that $g\cdot 1_R=0$. Then $\cdot$ is a (symmetric) partial action of $\Bbbk G[x,\sigma]$ on $R$, if and only if, $\cdot |_{\Bbbk G \otimes R}$ is a (symmetric) partial action of $\Bbbk G$ on $R$, $a\cdot (x \cdot 1_R)=(\sigma^{-1}(a)\cdot 1_R)(x \cdot 1_R)$, for all $a\in \Bbbk G$, and $\cdot$ coincides with \eqref{formula_ext}, where $q=\chi(g)$ and $w= x\cdot 1_R$.
    \end{thm}

\begin{proof}
    The first part follows directly of Theorem \ref{formula_eh_acao_parcial_geral} and Remark \ref{cond_ida}, and the converse follows by results \ref{formula_eh_acao_parcial_volta}, \ref{remark_cod} and \ref{lema_volta2}.
\end{proof}

\subsection{Rank one Hopf algebras} \label{Sec_Rank_one}

Let $\mathfrak{D}=(G,\chi,g,\beta)$ be a tuple such that $G$ is a finite group, $\chi$ is a character of $G$, $g \in Z(G)$ and $\beta \in \{0, 1\}\subseteq \Bbbk$.
The tuple $\mathfrak{D}$ is called a \textit{datum} if $\chi^d =1$  or $\beta (1 - g^d)=0$, where $d$ is the order of $\chi(g)$.
If $\beta (1 - g^d)=0$, we say that $\mathfrak{D}$ is of \emph{nilpotent type}; otherwise we say that $\mathfrak{D}$ is of \emph{non-nilpotent type}.

Given a datum $\mathfrak{D}$, the Hopf algebra $H_\mathfrak{D}$ is defined  as the algebra generated by $G$ and $x$ subject to the following relations (further than the relations of the group $G$):
\begin{align*}
	&x^d = \beta (1 - g^d), & & xh = \chi(h) hx, \, h\in G.
\end{align*}
The coalgebra structure is determined by $\Delta(x)=x \o 1 + g \o x$ and $\Delta(h)=h \o h,$ for all $h\in G$.
The Hopf algebras $H_\mathfrak{D}$ are known in the literature as the \textit{rank one Hopf algebras} \cite{rank_one}.
Observe that $H_\mathfrak{D}$ can be seen as the quotient of the Hopf-Ore extension $ \Bbbk G [x, \sigma]$ by the (Hopf) ideal $\langle x^d - \beta (1 - g^d)\rangle$.

\medbreak

Next, we determine all partial actions of rank one Hopf algebras $H_\mathfrak{D}$ on $R$ such that $g \cdot 1_R \in \{0, 1_R\}$.

Recall from \S \ref{Sec_quotient} that all such partial actions of $H_\mathfrak{D}$ come from partial actions of $\Bbbk G [x, \sigma]$, which in turn derive from partial actions of $\Bbbk G$.
Hence, results \ref{parte_global} and \ref{teo_carac_grupo} characterizes the extensions of these partial  actions of $\Bbbk G$ to $\Bbbk G [x, \sigma]$.
Moreover, Propositions \ref{prop:rank_one_nilp} and \ref{prop:rank_one_non_nilp} give us conditions for such partial actions of Hopf-Ore extension factorize through the quotient $H_\mathfrak{D}$.

Therefore, in light of what was developed, the next theorem holds.

\begin{thm}\label{thm_rank_one}
All partial actions of the rank one Hopf algebras on any algebra $R$ such that $g\cdot 1_R\in \{0, 1_R\}$ are determined in Table \ref{tabela_rank_one}. 
\begin{table}[H]
    \centering
    \caption{Partial actions of rank one Hopf algebras}
    \label{tabela_rank_one}
    \begin{tabular}{|c|c|c|}
    \hline
     & \begin{tabular}{c}
          Nilpotent type:  \\
          $I=\langle x^d \rangle $  
     \end{tabular} & \begin{tabular}{c}
          Non-nilpotent type:  \\
          $I=\langle x^d - 1 + g^d\rangle $  
     \end{tabular} \\ \hline
        $g\cdot 1_R = 1_R$ & \begin{tabular}{c}
             Lem. \ref{parte_global}  \\
             $\&$ \\
             Prop. \ref{prop:rank_one_nilp} (i)
        \end{tabular}  & \begin{tabular}{c}
          Thm. \ref{teo_carac_grupo}  \\
          $\&$ \\
          Prop. \ref{prop:rank_one_nilp} (ii) -- (iii)
     \end{tabular} \\ \hline
         $g\cdot 1_R = 0$ & \begin{tabular}{c}
             Lem. \ref{parte_global}  \\
             $\&$ \\
             Prop. \ref{prop:rank_one_non_nilp} (i)
        \end{tabular}  & \begin{tabular}{c}
          Thm. \ref{teo_carac_grupo}  \\
         $\&$ \\
          Prop. \ref{prop:rank_one_non_nilp} (ii)
     \end{tabular} \\ \hline
    \end{tabular}
    \end{table}
\end{thm}

In particular, we highlight two families of rank one Hopf algebras. 
Let $ n, d\geq 2$ such that $d | n,$ and $q $ a primitive $d$-th root of unity, consider the algebra generated by the letters $x$ and $g$
subject to the following relations:
\begin{align*}
&x^d = \beta (1 - g^d), & & g^n = 1, && xg = q gx,
\end{align*}
with $\beta \in \{0, 1\}$.
If $\beta = 0$, we have the generalized Taft algebra $H_{n,d} \, (q)$; otherwise, if $\beta =1, $ we have the Radford algebra $R_{n,d} \, (q)$.

Then, Theorem \ref{thm_rank_one} characterizes such partial actions of Hopf algebras
$H_{n,d} \, (q)$ and
$R_{n,d} \, (q)$.
Moreover, it also recovers the partial actions of Taft algebras determined in \cite[Thm. 3.8]{FMS2}.

\subsection{The particular case when $A=\mathbb{H}_4$}
\label{Sec_Sweedler}

In this section, we apply the developed results to characterize some partial actions of Hopf-Ore extensions of Sweedler's Hopf algebra $\mathbb{H}_4$ on an algebra $R$.
In the examples treated previously we usually deal with a Hopf-Ore extension of a group algebra, and in such case $g$ is a central element.
So, the importance of  dealing with $\mathbb{H}_4$ relies on the fact that it is the  smallest Hopf algebra that is not a group algebra and contains a grouplike element which is not central.

Consider a Hopf-Ore extension $\mathbb{H}_4[y, \sigma]$. First of all, we notice that any partial action $\cdot : \mathbb{H}_4[y, \sigma] \otimes R \to R$ such that $g \cdot 1_R = 1_R$ is actually a global action. Indeed, it follows by Proposition \ref{parte_global}, since  $\mathbb{H}_4[y, \sigma]$ is generated as algebra by $g, x$ and $y$, and both $x$ and $y$ are $(1,g)$-primitive elements.

The global actions of $\mathbb{H}_4$ on $R$ are characterized in \cite[Prop. 4.1]{Acoes_taft_Centrone}: an automorphism $\alpha_g$ of $R$ such that $\alpha_g^2=Id_R$, and an $\alpha_g$-derivation $\partial_x$ of $R$ such that $\alpha_g \partial_x = -\partial_x \alpha_g$ and $\partial_x^2 = 0$.

Hence, the global actions of $\mathbb{H}_4[y, \sigma]$ on $R$ are characterized by a global action of $\mathbb{H}_4$ plus an $\alpha_g$-derivation $\partial_y$ of $R$ such that $\partial_y \alpha_{g} = -\alpha_{g} \partial_y$ and $\partial_y \partial_{x} = -\partial_{x} \partial_y $.

\medbreak

Now, we characterize the partial actions $\mathbb{H}_4[y, \sigma]$ on $R$ such that $g \cdot 1_R = 0$.
We begin with some general lemmas.

\begin{lem}\label{acao_skew_primitivo_extensao}
    Let $\cdot : \mathcal{H} \otimes R \to R$ be a partial action, $g \in G(\mathcal{H})$ and $x\in P_{1,g}(\mathcal{H})$ such that $xg=qgx$, for some $q \in \Bbbk^{\times}$, and $g \cdot 1_R =0$.  Consider $\mathcal{H} [y,\sigma]$ and the linear map defined in \eqref{formula_ext}.
    Then, for any $b \in \mathcal{H}$:
    \begin{itemize}
    \item[(i)] \eqref{somatorio_acao} holds for $a \in \{1, g, x\}$,
    \item[(ii)] \eqref{somatorio_acao_simetrica} holds for $a \in \{1, g^{-1}, g^{-1}x\}$.
\end{itemize}
\end{lem}

\begin{proof} Initially, by Lemma \ref{acao_produto_de_skew_primitivos_nichols}, $x \cdot r = \Omega r$ and $g^{-1}x \cdot r = -qr \Omega$, where $\Omega := x\cdot 1_R \in R$. Moreover, $g\cdot r=g^{-1}\cdot r=0$, for all $r\in R$. 

$(i)$ The equation \eqref{somatorio_acao} is trivially satisfied for $a=1$ or $a=g$.
Now, for $a=x$, as $\Delta_2(x)=x\otimes 1\otimes 1+ g\otimes\Delta(x)$, for all $b\in \mathcal{H}$ we have: 
  \begin{eqnarray*}
     & &   \sum_{k=0}^j (-1)^k q^{ -\frac{k(k-1)}{2}} { j \choose k }_{q^{-1}} (a_1\cdot w^{j-k})(a_2g^{k}b\cdot r)(a_3\cdot w^k)\\
 & =    &   \sum_{k=0}^j (-1)^k q^{ -\frac{k(k-1)}{2}} { j \choose k }_{q^{-1}} (x\cdot w^{j-k})(g^{k}b\cdot r)w^k \\
  & =  & \sum_{k=0}^j (-1)^k q^{ -\frac{k(k-1)}{2}} { j \choose k }_{q^{-1}} \Omega w^{j-k}(g^{k}b\cdot r)w^k \\
    & = & \sum_{k=0}^j (-1)^k q^{ -\frac{k(k-1)}{2}} { j \choose k }_{q^{-1}} (a_1\cdot 1_R) w^{j-k}(g^{k}\sigma^{-j}(a_2)b\cdot r)w^k.
     \end{eqnarray*}
     
$(ii)$ Analogous to item $(i)$.
\end{proof}

\begin{rem}\label{lema_parcentarl}
   Let $R$ be an algebra and $\Omega, w \in R$.
   If $w \Omega + \Omega w \in Z(R)$, then $w^{2 \ell} \Omega = \Omega w^{2\ell}$, for all $\ell \in \N_0$.
Indeed, for any $r \in R$, $r(w \Omega + \Omega w) = (w \Omega + \Omega w)r$, since $w \Omega + \Omega w \in Z(R)$.
Then, for $r = w$, it holds that $w^2 \Omega = \Omega w^2$, and so the result follows by induction on $\ell$.
\end{rem}

\begin{lem}\label{acao_skew_primitivo_extensao2}
      Let $\cdot : \mathcal{H} \otimes R \to R$ be a partial action, $g\in G(\mathcal{H})$ and $x \in P_{1,g}(\mathcal{H})$ such that $g^2=1$, $x^2=0$ and $gx=-xg$, and $g \cdot 1_R =0$. 
    Consider $\mathcal{H} [y,\sigma]$, where $q=-1$, and the linear map defined in \eqref{formula_ext}.
     Then, the following assertions are equivalent:
    \begin{itemize}
    \item[(i)] \eqref{somatorio_acao} holds for $a=gx$ and $b \in \{1, g,  x, gx \}$,
    \item[(ii)] \eqref{somatorio_acao_simetrica} holds for $a=x$ and $b \in \{1, g,  x, gx \}$,
    \item[(iii)] $ (x\cdot 1_R) w + w(x\cdot 1_R) \in Z(R)$.
    \end{itemize}
\end{lem}

\begin{proof} By Lemma \ref{acao_produto_de_skew_primitivos_nichols}, $x \cdot r = \Omega r$ and $gx \cdot r = r \Omega$, where $\Omega := x\cdot 1_R \in R$. Moreover, $g\cdot r=0$, for all $r\in R$ and $\Omega^2 \in Z(R)$.

Let's check that $(i)$ if only if $(iii)$.
As $q=-1$, $gx\in P_{g,1}(\mathcal{H})$ and $\Delta_2(gx)=gx\otimes g\otimes g +1\otimes gx\otimes g+1\otimes 1\otimes gx$, for $a=gx$ in \eqref{somatorio_acao} one gets the equivalent expression 
  \begin{align}\label{eq_sweedler}
  \begin{split}
          &   \sum_{k=0}^j (-1)^{\frac{k(3-k)}{2}} { j \choose k }_{-1}  w^{j-k}(g^{k}b\cdot r)w^k\Omega \, = \\
       & \sum_{k=0}^j (-1)^{\frac{k(3-k)}{2}} { j \choose k }_{-1} \hspace{-0.2cm}\left((-1)^j\Omega w^{j-k}(g^{k+1}b\cdot r) + w^{j-k}(g^{k+1}xb\cdot r) \right) \! w^k.
  \end{split}
   \end{align}

In particular, for $j=1$ and $b=g$, we obtain $-rw\Omega  =  - \Omega wr-w\Omega r+r\Omega w$, that is, 
\begin{eqnarray*}
 r(\Omega w + w\Omega) & =& (\Omega w + w\Omega) r,
 \end{eqnarray*}
 for all $r\in R$. Therefore, $\Omega w + w\Omega\in Z(R)$.
 
Conversely, assume $\Omega w + w\Omega\in Z(R)$.
Then, by Remark \ref{lema_parcentarl}, $w^{2 \ell} \Omega = \Omega w^{2\ell}$, for all $\ell \in \N_0$.
 Moreover, since $q=-1$, if $j$ is an even number, ${ j \choose k }_{-1}=0$ for all $k$ odd (see \S \ref{Sec_qcomb}), and so equality \eqref{eq_sweedler} can be rewritten as 
 \begin{align*}
  \begin{split}
          &   \sum_{\substack{k=0 \\ k \ \textrm{even}}}^j (-1)^{\frac{k(3-k)}{2}} { j \choose k }_{-1}  w^{j-k}(b\cdot r)w^k\Omega \, = \\
       & \sum_{\substack{k=0 \\ k \ \textrm{even}}}^j (-1)^{\frac{k(3-k)}{2}} { j \choose k }_{-1} \hspace{-0.2cm}\left(\Omega w^{j-k}(gb\cdot r) + w^{j-k}(gxb\cdot r) \right) \! w^k.
  \end{split}
   \end{align*}
 Then, it is straightforward to check that $$w^{j-k}(b\cdot r)w^k\Omega = \left(\Omega w^{j-k}(gb\cdot r) + w^{j-k}(gxb\cdot r) \right)\! w^k,$$ holds for $b \in \{1, g,  x, gx \}$, and so \eqref{eq_sweedler}.

Now, assume $j$ an odd number.
To show that \eqref{eq_sweedler} holds for each $b \in \{1, g,  x, gx \}$, we proceed by cases.
Recall that $g\cdot r=0$.

 If $b=1$, then \eqref{eq_sweedler} means
  \begin{eqnarray*}
      & & \sum_{\substack{k=0 \\ k \ \textrm{even}}}^j (-1)^{\frac{k(3-k)}{2}} { j \choose k }_{-1}  w^{j-k}r w^k\Omega\\
      & = & \sum_{\substack{k=0 \\ k \ \textrm{odd}}}^j (-1)^{\frac{k(3-k)}{2}} { j \choose k }_{-1} \hspace{-0.2cm}\left(-\Omega w^{j-k}r  + w^{j-k}  \Omega r \right) \! w^k\\
      & & +  \sum_{\substack{k=0 \\ k \ \textrm{even}}}^j (-1)^{\frac{k(3-k)}{2}} { j \choose k }_{-1} \hspace{-0.2cm} w^{j-k}r\Omega  w^k,
   \end{eqnarray*}
which is true since $w^k\Omega=\Omega w^k$ ($k$ even) and $w^{j-k}\Omega=\Omega w^{j-k}$ ($k$ odd).

  If $b=g$, then \eqref{eq_sweedler} means 
 \begin{eqnarray*}
          &&   \sum_{\substack{k=0 \\ k \ \textrm{odd}}}^j (-1)^{\frac{k(3-k)}{2}} { j \choose k }_{-1}  w^{j-k}rw^k\Omega\\
      &= & \sum_{\substack{k=0 \\ k \ \textrm{even}}}^j (-1)^{\frac{k(3-k)}{2}} { j \choose k }_{-1} \hspace{-0.2cm}\left(-\Omega w^{j-k}r - w^{j-k}\Omega r\right) \! w^k \\
     &&-  \sum_{\substack{k=0 \\ k \ \textrm{odd}}}^j (-1)^{\frac{k(3-k)}{2}} { j \choose k }_{-1} w^{j-k}r \Omega w^k,
   \end{eqnarray*}
that is equivalent to
 \begin{align*}
          &   \sum_{\substack{k=0 \\ k \ \textrm{odd}}}^j (-1)^{\frac{k(3-k)}{2}} { j \choose k }_{-1}  w^{j-k}rw^{k-1}\left(w\Omega + \Omega w\right)\\
      = & - \sum_{\substack{k=0 \\ k \ \textrm{even}}}^j (-1)^{\frac{k(3-k)}{2}} { j \choose k }_{-1} \left(\Omega w + w\Omega \right) w^{j-k-1} rw^k.
   \end{align*}

By a change of variable, the last equality becomes
    \begin{align*}
  \begin{split}
          &   - \sum_{\substack{k=0 \\ k \ \textrm{even}}}^j (-1)^{\frac{k(1-k)}{2}} { j \choose k+1 }_{-1} w^{j-k-1}rw^{k}\left(w\Omega + \Omega w\right)\\
      = & - \sum_{\substack{k=0 \\ k \ \textrm{even}}}^j (-1)^{\frac{k(3-k)}{2}} { j \choose k }_{-1} \left(\Omega w + w\Omega \right) w^{j-k-1} rw^k,
  \end{split}
   \end{align*}
   which is true since $\Omega w + w\Omega\in Z(R)$, and ${ j \choose k+1 }_{-1}= { j_D \choose k_D }={ j \choose k}_{-1}$ for $0\leq k \leq j$ when $j$ is odd and $k$ is even (see \S \ref{Sec_qcomb}).

   Similarly, one checks for $b=x$ and $b=gx$, using $\Omega^2\in Z(R)$.
   Therefore, \eqref{eq_sweedler}, and consequently \eqref{somatorio_acao}, holds for $a=gx$ and $b \in \{1, g,  x, gx \}$.
   It concludes $(i) \Longleftrightarrow (iii)$.
   Analogously we obtain $(ii) \Longleftrightarrow (iii)$.
\end{proof}

Then, we characterize some partial actions of a Hopf-Ore extensions of $\mathbb{H}_4$ on $R$, as follows.

\begin{prop}\label{carac_acoes_extensoes_Ore_Sweedler}
    Let $\cdot : \mathbb{H}_4 \otimes R \to R$ be a partial action such that $g \cdot 1_R=0.$ 
    Then the linear map $\cdot$ defined in \eqref{formula_ext} is a symmetric partial action of $\mathbb{H}_4[y,\sigma]$ on $R$ if and only if $w \Omega + \Omega w \in Z(R)$, where $x \cdot 1_R := \Omega$ and $y \cdot 1_R := w$.
\end{prop}

\begin{proof}
It follows from Lemmas \ref{acao_skew_primitivo_extensao} and \ref{acao_skew_primitivo_extensao2}.
\end{proof}

\subsection{Nichols Hopf algebras} \label{Sec_Nichols}

In this section, we apply the developed results to describe some examples of partial actions of $\mathbb{H}_{2^n}$ on $R$.
The results presented here are already known in the literature \cite[\S 4.1]{FMS2}, but we describe how to get them in this context of Hopf-Ore extensions and quotients, inductively.

\medskip

We start recalling the family of Nichols Hopf algebras $\mathbb{H}_{2^n}$.
They were originally introduced in \cite{Taft}, but were named after the work of Nichols \cite{nichols}.
For convenience, we present them as in \cite[\S 2.2]{etingof}. 

Let $n \geq 2$ be an integer.
As algebra, $\mathbb{H}_{2^n}$ is generated by the letters $g,x_1, \cdots, x_{n-1}$ subject to the relations 
\begin{align*}
&g^2=1,  & &x_i g = -g x_i, &&x_ix_j = -x_jx_i,
\end{align*}
for all $i, j \in \I_{n-1}$.
Its coalgebra structure is determined by $g\in G(\mathbb{H}_{2^n})$ and $x_i\in P_{1,g}(\mathbb{H}_{2^n})$.

\smallbreak

Note that such Hopf algebras can be obtained by a iterated process of Hopf-Ore extensions and quotients as follows.

Let $\mathbb{H}_{2^2} = \mathbb{H}_{4} = \Bbbk \langle g, x_1 \ | \ g^2 =1, x_1^2 = 0, x_1 g = - g x_1 \rangle$ be the Sweedler's Hopf algebra; for $n \geq 2$ one obtains inductively
$\mathbb{H}_{2^{n+1}} := \mathbb{H}_{2^n}[x_n,\sigma_n]/\langle x_n^2 \rangle$, where $\sigma_n: \mathbb{H}_{2^n} \to \mathbb{H}_{2^n}$ is the automorphism given by $\sigma_n(g)=-g$ and
$\sigma_n(x_i)=-x_i$, for $i \in \I_{n-1}$.

\smallbreak

At this first moment, we note that any partial action $\cdot : \mathbb{H}_{2^{n}} \otimes R \to R$ such that $g \cdot 1_R = 1_R$ is a global action.
Since $\mathbb{H}_{2^{n}}$ is generated as algebra by $g, x_1, \dots, x_{n-1}$, and each $x_i$ is a $(1,g)$-primitive element, we get it by Lemma \ref{parte_global}.
Namely, the grouplike element $g$ acts by an automorphism $\alpha_g$ of $R$ such that $\alpha_g^2=Id_R$, and each skew-primitive element $x_i$ acts by an $\alpha_g$-derivation $\partial_i$ of $R$ such that $\alpha_g \partial_i = -\partial_i \alpha_g$ and $\partial_i \partial_j = -\partial_j \partial_i$, for $i, j \in \I_{n-1}$.
 In particular, any map $\partial_i$ is 2-nilpotent.

Consequently,
we can extend such action to a partial (indeed global) action of $\mathbb{H}_{2^n}[x_n, \sigma_n]$ on $R$, where the skew-primitive element $x_n$ acts by any $\alpha_g$-derivation $\partial_n$ of $R$ satisfying analogous relations to the previous ones, \emph{i.e.}, such that $\alpha_g \partial_n = -\partial_n \alpha_g$ and $\partial_n \partial_j = -\partial_j \partial_n$, for any $j \in \I_{n-1}$.

Finally, if we have that $\partial_n$ is $2$-nilpotent, then this action factorizes through $ \mathbb{H}_{2^n}[x_n, \sigma_n]/\langle x_n^2 \rangle$ and it is a global action of $\mathbb{H}_{2^{n+1}}$ on $R$. 

\smallbreak

Next, our efforts are to obtain the class of symmetric partial actions $\mathbb{H}_{2^{n+1}}$ on $R$ such that $g \cdot 1_R = 0$ and $x_i \cdot 1_R \in Z(R)$, for $i \in \I_n$, through an iterated process. 

First, consider a (symmetric) partial action of $\mathbb{H}_{4}$ on $R$ given by $g \cdot 1_R =0$ and $x \cdot 1_R = w_1 \in Z(R)$,
and $\mathbb{H}_{2^2}[x_2, \sigma_2]$.
By Proposition \ref{carac_acoes_extensoes_Ore_Sweedler}, any choice $w_2 \in Z(R)$ defines a symmetric partial action of $\mathbb{H}_{2^2}[x_2, \sigma_2]$ on $R$ through expression \eqref{formula_ext}, putting $x_2 \cdot 1_R := w_2$.
Moreover, by Lemma \ref{lemma_truncamento}, $x_2^j a \cdot r =0$ for all $j \geq 2$, and so this partial action induces a symmetric partial action of $\mathbb{H}_{2^3}= \mathbb{H}_{2^2}[x_2, \sigma_2]/ \langle x_2^2 \rangle$ on $R$.
Furthermore, by Lemmas \ref{acao_produto_de_skew_primitivos_nichols} and \ref{acao_produto_de_skew_primitivos_nichols2}, it is easy to conclude that such symmetric partial action satisfies $g \cdot 1_R = 0$, $x_i \cdot r  = g x_i \cdot r = w_i r$, for $i \in \I_{2}$, and $x_1 x_2 \cdot r  = g x_1 x_2 \cdot r =0$.

\smallbreak

Now, let $n \geq 2$ and $\cdot : \mathbb{H}_{2^n} \otimes R \to R$ be a symmetric partial action given by $g \cdot 1_R = 0$ and $x_i \cdot 1_R := w_i \in Z(R)$, for $i \in \I_{n-1}$.
Again, by Lemmas \ref{acao_produto_de_skew_primitivos_nichols} and \ref{acao_produto_de_skew_primitivos_nichols2}, we also conclude $x_i \cdot r  = g x_i \cdot r = w_i r$ and $g^\ell x_{i_1} x_{i_2} ... x_{i_s} \cdot r =0$, for $\ell \in \{0,1\},$ $s\geq 2$ and $i, i_k \in \I_{n-1}$.

To extend this partial action of $\mathbb{H}_{2^n}$ on $R$ to a symmetric partial action of $\mathbb{H}_{2^n}[x_n, \sigma_n]$ on $R$ via \eqref{formula_ext}, we have to investigate if conditions \eqref{cod_volta} and \eqref{cod_voltasym} hold for every $a, b \in \mathbb{H}_{2^n}, r \in R$ and $j \in \N$.
Consider  $w_n := x_n \cdot 1_R  \in Z(R)$; by Lemma \ref{lemma_truncamento}, it follows that $x_n^j b \cdot r =0$ for $j \geq 2$ and, consequently, the above conditions are valid for $j \geq 2$, since $ax_n^j  = x_n^j \sigma_n^{-j}(a)$.
Hence, we just need to verify them for $j=1$.

To do so, we present a similar result to Lemma \ref{acao_skew_primitivo_extensao2} as follows:
\begin{lem}\label{acao_skew_primitivo_extensao2_Nicohls}
      Let $\cdot : \mathbb{H}_{2^n} \otimes R \to R$ be a partial action such that $g \cdot 1_R =0$ and $x_i\cdot 1_R\in Z(R)$, for $i \in \I_{n-1}$. 
    Consider $\mathbb{H}_{2^n}[x_n, \sigma_n]$ and the linear map defined in \eqref{formula_ext}. Then,  the following assertions are equivalent:
         \begin{itemize}
    \item[(i)] \eqref{somatorio_acao} holds for $j=1$, $a=gx_i$ and $b \in \mathbb{H}_{2^n}$,
    \item[(ii)] \eqref{somatorio_acao_simetrica} holds for $j=1$, $a=x_i$ and $b \in \mathbb{H}_{2^n}$,
    \item[(iii)] $ (x_i\cdot 1_R) (x_n\cdot 1_R) \in Z(R)$.
    \end{itemize}
    \end{lem}
\begin{proof}
   Analogous to the proof of Lemma \ref{acao_skew_primitivo_extensao2}.
\end{proof}

Now, we check \eqref{cod_volta} for $j=1$ proceeding by cases.
First, by Lemma \ref{acao_skew_primitivo_extensao}, it holds for $a=1$, $a=g$ or $a=x_i,$ for $i \in \I_{n-1}$, and any $b \in \mathbb{H}_{2^n}$.
Moreover, since $w_i \in Z(R)$ for all $i \in \I_{n}$, by Lemma \ref{acao_skew_primitivo_extensao2_Nicohls}  it also holds for $a=gx_i$, $i \in \I_{n-1}$, and any  $b \in \mathbb{H}_{2^n}$.
Therefore, all that remains is to check the condition for $a = g^\ell x_{i_1}...x_{i_s}$, where $\ell \in \{0,1\}$ and $s \geq 2$, and any $b \in \mathbb{H}_{2^n}$. 
 
One the one hand,
    \begin{equation}\label{cond_1_nic}
    \begin{split}
       a \cdot (x_n b \cdot r) & = a \cdot \left( w_n(b \cdot r) - (gb\cdot r)w_n \right) \\
       & = (a_1 \cdot w_n) (a_2b \cdot r) - (a_1gb \cdot r) (a_2 \cdot w_n),
    \end{split}
        \end{equation}
where  last equality follows from the symmetry of the partial action of $\mathbb{H}_{2^n}$ on $R$.
On the other hand, 
       \begin{equation}\label{cond_2_nic}
\begin{split}
        & (a_1 \cdot 1_R)(a_2 x_n b \cdot r) \\
        = & \ (a_1 \cdot 1_R)(x_n \sigma_n^{-1}(a_2)b \cdot r) \\
        = & \ (a_1 \cdot 1_R)\left(w_n(\sigma_n^{-1}(a_2)b \cdot r) - (g\sigma_n^{-1}(a_2)b\cdot r)w_n \right) \\
        = & \ (a_1 \cdot 1_R)w_n(\sigma_n^{-1}(a_2)b \cdot r)  - (a_1 \cdot 1_R) (g\sigma_n^{-1}(a_2)b\cdot r)w_n.
\end{split}
    \end{equation}

First, if $a =x_{i_1}x_{i_2}$, then
	\begin{align*}
		\Delta( x_{i_1}x_{i_2}) =  x_{i_1}x_{i_2} \o 1 - g x_{i_1} \o  x_{i_2}   
   + g x_{i_2} \o x_{i_1}  
		 + 1 \o x_{i_1}x_{i_2},
	\end{align*}
and so \eqref{cond_1_nic} and \eqref{cond_2_nic} mean, respectively,
\begin{align*}
    & (a_1 \cdot w_n) (a_2b \cdot r) - (a_1gb \cdot r) (a_2 \cdot w_n) \\
    = &    - ( g x_{i_1} \cdot w_n) ( x_{i_2} b \cdot r) + ( g x_{i_2} \cdot w_n) ( x_{i_1} b \cdot r)  +  ( g x_{i_1} gb \cdot r) ( x_{i_2} \cdot w_n) \\
    & - ( g x_{i_2} gb \cdot r) ( x_{i_1} \cdot w_n)  \\
    = & - w_{i_1} w_n ( x_{i_2} b \cdot r) + w_{i_2} w_n ( x_{i_1} b \cdot r) - (x_{i_1}b \cdot r) w_{i_2} w_n + (x_{i_2} b \cdot r) w_{i_1} w_n \\
    = & \ 0,
\end{align*}
and
\begin{align*}
   & (a_1 \cdot 1_R)w_n(\sigma_n^{-1}(a_2)b \cdot r)  - (a_1 \cdot 1_R) (g\sigma_n^{-1}(a_2)b\cdot r)w_n \\
   = & - (g x_{i_1} \cdot 1_R)w_n(\sigma_n^{-1}(x_{i_2})b \cdot r)  + (g x_{i_2} \cdot 1_R)w_n(\sigma_n^{-1}(x_{i_1})b \cdot r)  \\
   & + (gx_{i_1} \cdot 1_R) (g\sigma_n^{-1}(x_{i_2})b\cdot r)w_n  - (gx_{i_2} \cdot 1_R) (g\sigma_n^{-1}(x_{i_1})b\cdot r)w_n \\
   = & \ w_{i_1} w_n ( x_{i_2} b \cdot r)  - w_{i_2} w_n ( x_{i_1}b \cdot r) - w_{i_1} ( gx_{i_2} b \cdot r) w_n +  w_{i_2} (gx_{i_1} b \cdot r) w_n.
\end{align*}

Hence, \eqref{cod_volta} holds for $a = x_{i_1}x_{i_2}$ if and only if  
$$w_{i_1} w_n ( x_{i_2} b \cdot r) + w_{i_2} (gx_{i_1} b \cdot r) w_n = w_{i_2} w_n ( x_{i_1}b \cdot r) + w_{i_1} ( gx_{i_2} b \cdot r) w_n,$$
for any $b \in \mathbb{H}_{2^n}, r \in R$.
But, it is an easy computation to conclude that indeed $w_i (x_j b \cdot r) = w_i (g x_j b \cdot r) = w_j (x_i b \cdot r) = w_j (gx_i b \cdot r)$. Analogously, the same conclusion holds for $a = gx_{i_1}x_{i_2}$.

Now, consider $a=g^\ell x_{i_1}x_{i_2}...x_{i_s}$, where $\ell \in \{0,1\}$ and $s \geq 3$, and write $a = x_{i_1}x_{i_2}x_{i_3}y$, for some $y \in \mathbb{H}_{2^n}$. 
	Thus,
	\begin{eqnarray*}
		\Delta(a) & = &  x_{i_1}x_{i_2}x_{i_3}y_1 \o y_2 - g x_{i_1}x_{i_3}y_1 \o  x_{i_2}y_2 \\
		& & + \ g  x_{i_2}x_{i_3}y_1 \o x_{i_1}y_2 +  x_{i_3}y_1 \o x_{i_1}x_{i_2}y_2 \\
		& & + \ g x_{i_1}x_{i_2}y_1 \o x_{i_3} y_2  + x_{i_1}y_1 \o x_{i_2}x_{i_3} y_2 \\
		& & - \ x_{i_2} y_1 \o x_{i_1}x_{i_3} y_2 + g  y_1 \o x_{i_1}x_{i_2}x_{i_3} y_2.
	\end{eqnarray*}

Then, clearly both expressions \eqref{cond_1_nic} and \eqref{cond_2_nic} vanish, and so \eqref{cod_volta} holds, for any $b \in \mathbb{H}_{2^n}, r \in R$.	

To prove the symmetric condition \eqref{cod_voltasym}, one proceeds analogously.

Therefore, any choice of $w_n \in Z(R)$ defines a symmetric partial action of $\mathbb{H}_{2^n}[x_n, \sigma_n]$ on $R$ through the expression \eqref{formula_ext}, putting $x_n \cdot 1_R := w_n$.
At last, again by Lemma \ref{lemma_truncamento}, this action induces a symmetric partial action of the quotient $\mathbb{H}_{2^{n+1}}= \mathbb{H}_{2^n}[x_n, \sigma_n]/ \langle x_n^2 \rangle$ on $R$.

\smallbreak

Hence, in the above discussion, we reobtain the characterization of some partial actions of Nichols Hopf algebras on algebras given in \cite[\S 4.1]{FMS2}, but in the context of Hopf-Ore extensions and quotients.

\medbreak

\noindent{\textbf{Remark.}} We end the paper highlighting that the techniques and approach developed here can be used to obtain examples of partial actions for other families of Hopf algebras (mainly for the ones generated by grouplike and skew primitive elements).
In particular, we recover some examples of partial actions of pointed Hopf algebras of dimensions 8 and 16 over their base fields \cite{corresponding}.

\medbreak

\noindent{\textbf{Acknowledgments.}}
The authors thank Alveri Sant'Ana and Marcelo Muniz Alves for interesting suggestions, conversations and advice at different
moments of our research.

\bibliographystyle{abbrv}

\bibliography{referencias}

\end{document}